\documentclass[a4paper] {article}
[12pt]

\makeatletter
\renewcommand*\l@section{\@dottedtocline{1}{1.5em}{2.3em}}
\makeatother

\usepackage{amsfonts}
\usepackage{amssymb}
\usepackage[T1]{fontenc}

\usepackage{tikz}
\usetikzlibrary{calc}

\usepackage{CJK}
\usepackage{amsmath}
 
\usepackage{amsfonts}
\usepackage{amssymb}
\usepackage{amsthm}
\usepackage{amssymb}
\usepackage{enumerate}
\usepackage[calc]{picture}
\usepackage[all,cmtip]{xy}

\usepackage[mathscr]{eucal}
\usepackage{eqlist}

\usepackage{color}
\usepackage{abstract} 
\usepackage[T1]{fontenc}
 
 \usepackage[top=3.0 cm, bottom=3.0cm, left=3.5cm, right= 3.5cm]{geometry}    

\usepackage{cite}

\theoremstyle{plain}
\newtheorem{theorem}{Theorem}[section]
\newtheorem{proposition}[theorem]{Proposition}
\newtheorem{lemma}[theorem]{Lemma}
\newtheorem{corollary}[theorem]{Corollary}

\theoremstyle{definition}
\newtheorem{definition}{Definition}[section]

\theoremstyle{example}
\newtheorem{example}{Example}[section]

 \theoremstyle{remark}%{myrem}
 \newtheorem{remark}{Remark}[section]
\usepackage{etoolbox}
\numberwithin{equation}{section}
\numberwithin{theorem}{section}

\begin{document}

\begin{center}
{\Large {\textbf {A Discrete Morse Theory for %The Embedded Homology of 
Hypergraphs }}}
 \vspace{0.58cm}
 
 Shiquan Ren*, Chong Wang*,  Chengyuan Wu*, Jie Wu*

\bigskip

\bigskip
    
    \parbox{24cc}{{\small

{\textbf{Abstract}.}  
%Hypergraphs are higher-dimensional generalizations of graphs.  An (abstract) simplicial complex is a  hypergraph such that all the faces of hyperedges are still hyperedges.  The topology of a hypergraph  can be investigated by its embedded homology, the homology of its associated complex, and the homology of its lower-associated complex. 
%In this paper,  
%we study the  evolutions of hypergraphs as well as  the induced homomorphisms of the  embedded homology groups.  % of hypergraphs. %Moreover, we indicate that persistent homology induced by evolution sequences of hypergraphs has potential applications    in  social networks. 
%The main result of this paper investigates the homomorphisms of the  embedded homology induced from evolutions of hypergraphs  and the relations with the  lower-associated simplicial complexes. %influence of the evolutions of hypergraphs on the embedded homology.  
A hypergraph can be obtained from a simplicial complex by deleting some non-maximal simplices.   By \cite{parks},  a hypergraph gives an associated simplicial complex.  By \cite{h1},  the embedded homology of a hypergraph is the homology of the infimum chain complex, or equivalently, the homology of the supremum chain complex.   In this paper,  we generalize the discrete Morse theory for simplicial complexes by R. Forman \cite{ forman1,forman2,forman3}    and  give a discrete Morse theory for hypergraphs. We use the critical simplices of the associated simplicial complex  to construct a sub-chain complex of the infimum chain complex and a sub-chain complex of the supremum chain complex,  then prove that the embedded homology of a hypergraph is isomorphic to the homology of the  constructed chain complexes.  Moreover,  we  define discrete Morse functions on hypergraphs and compute the embedded homology in terms of the critical hyperedges.  As by-products, we derive some Morse inequalities and collapse results for hypergraphs. 
}}
\end{center}

\vspace{1cc}



\footnotetext[3] { * first authors. }

\section{Introduction}

Hypergraph is an important model for complex networks, for example, the collaboration network.  An edge in a graph consists of two vertices, while a hyperedge in a hypergraph allows multiple vertices (cf. \cite{berge}).  In topology, a hypergraph can be obtained from a simplicial complex by deleting some non-maximal simplices  (cf. \cite{h1,parks}). %admitting that the faces of a simplex not to be a simplex. 

\smallskip

 Let $V_\mathcal{H}$ be a totally-ordered finite set. Let $2^V$ denote the powerset of $V$.  Let $\emptyset$ denote the empty set.  A {\it hypergraph}  is a pair $(V_\mathcal{H},\mathcal{H})$ where $\mathcal{H}$ is a subset of  $2^V\setminus\{\emptyset\}$ (cf. \cite{berge,parks}).  An element of $V_\mathcal{H}$ is called a {\it vertex}. % and an element of $\mathcal{H}$ is called a {\it hyperedge}. 
 Let $k\geq 0$.  We call a hyperedge $\sigma\in\mathcal{H}$ consisting of $k+1$ vertices   a  {\it $k$-dimensional hyperedge},  or  a {\it hyperedge} of  {\it dimension $k$}, and denote $\sigma$ as $\sigma^{(k)}$.    For any hyperedges $\sigma,\tau\in\mathcal{H}$, if $\sigma$ is a  proper subset  of $\tau$, then we write   $\sigma<\tau$ or $\tau>\sigma$.  %Given a hyperedge $\sigma\in \mathcal{H}$, we define the {\bf boundary} of $\sigma$,  denoted as $\partial\sigma$, as the collection of all nonempty proper subsets of $\sigma$.   
Throughout this paper, we assume that each vertex in $V_\mathcal{H}$ appears in at least one hyperedge in $\mathcal{H}$.  Hence $V_\mathcal{H}$ is the union $\bigcup_{\sigma\in\mathcal{H}} \sigma$, and we simply denote a hypergraph $(V_\mathcal{H},\mathcal{H})$ as $\mathcal{H}$.  
Given two hypergraphs $\mathcal{H}$ and $\mathcal{H}'$, if each hyperedge of $\mathcal{H}$ is also a hyperedge of $\mathcal{H}'$,  then we write $\mathcal{H}\subseteq\mathcal{H}'$ and say that $\mathcal{H}$ can be {\it embedded in} $\mathcal{H}'$.  

An {\it (abstract) simplicial complex}  is a  hypergraph satisfying the following  condition (cf. \cite[p. 107]{hatcher}, \cite[Section~1.3]{wu1}):
%\begin{enumerate}[(i).]
%\item
%for any $v\in V_\mathcal{H}$, the single-point set $\{v\}$ is in  $\mathcal{H}$; 
%\item
 for any $\sigma\in \mathcal{H}$ and any non-empty subset $\tau\subseteq \sigma$,     $\tau$ must be a hyperedge in  $\mathcal{H}$. 
%\end{enumerate}
 A hyperedge  of a simplicial complex is called a {\it simplex}.  %for any $v\in V_\mathcal{H}$, the single-point set $\{v\}$ is in  $\mathcal{H}$. %We denote the boundary maps of simplicial complexes as $\partial_*$, $*=0,1,2,\ldots$  

\smallskip

Let $\mathcal{H}$ be a hypergraph.    Let $R$  be a commutative ring with multiplicative unity.  Let $n$ be a nonnegative integer. Let $R(\mathcal{H})_n$  be the collection of all the formal linear combinations of  the  $n$-dimensional hyperedges of $\mathcal{H}$, with coefficients in $R$.  In particular, if $\mathcal{H}$ is a simplicial complex,  then we also denote $R(\mathcal{H})_n$ as $C_n(\mathcal{H};R)$.  

We consider a 
 simplicial complex $\mathcal{K}$ such that $\mathcal{H}\subseteq\mathcal{K}$. %that contains all the hyperedges of $\mathcal{H}$ as simplices.  
 Let $\partial^{\mathcal{K}}_*$ be the boundary maps of $\mathcal{K}$.  By \cite{h1},  as sub-chain complexes of $\{C_n(\mathcal{K};R),\partial_n^{\mathcal{K}}\}_{n\geq 0}$,  the {\it infimum chain complex} of $\mathcal{H}$ is defined by
   \begin{eqnarray}\label{eq-0.01}
   \text{Inf}_n(R(\mathcal{H})_*)=R(\mathcal{H})_n\cap (\partial^{\mathcal{K}}_n)^{-1}R(\mathcal{H})_{n-1}, \text{\ \ \ } n\geq 0;  
   \end{eqnarray}
   and the {\it supremum chain complex} of $\mathcal{H}$ is defined by
   \begin{eqnarray}\label{eq-0.02}
   \text{Sup}_n(R(\mathcal{H})_*)=R(\mathcal{H})_n+ \partial^{\mathcal{K}}_{n+1}R(\mathcal{H})_{n+1}, \text{\ \ \ } n\geq 0. 
   \end{eqnarray}
Here $(\partial^{\mathcal{K}}_n)^{-1}$ denotes the pre-image.   It is proved in \cite{h1} that both (\ref{eq-0.01}) and (\ref{eq-0.02}) do not depend on the choice of $\mathcal{K}$.  Hence we may choose $\mathcal{K}$ as the smallest simplicial complex  %containing all the hyperedges of $\mathcal{H}$, 
   such that $\mathcal{H}\subseteq\mathcal{K}$, which is called the {\it associated simplicial complex} of $\mathcal{H}$ (cf. \cite{parks})  and is denoted as $\Delta\mathcal{H}$ in this paper.

By \cite{h1}, the homology of  (\ref{eq-0.01}) and (\ref{eq-0.02}) are isomorphic.    The {\it embedded homology} of $\mathcal{H}$, denoted as  $H_*(\mathcal{H};R)$, is defined as the homology of (\ref{eq-0.01}) and  (\ref{eq-0.02}).  
 In particular, if $\mathcal{H}$ is a simplicial complex, then both (\ref{eq-0.01}) and  (\ref{eq-0.02})  are $C_n(\mathcal{H};R)$, and $H_*(\mathcal{H};R)$ is the usual homology of the simplicial complex.

\smallskip

In   \cite{ forman1,forman2,forman3}, R. Forman has developed a discrete Morse theory for simplicial complexes (and general cell complexes).   He defined discrete Morse functions $f$ on simplicial complexes by assigning a real number $f(\sigma)$ to each simplex $\sigma$ such that  for any $n\geq 0$ and any simplex $\alpha^{(n)}$, there exist at most one $\beta^{(n+1)}>\alpha^{(n)}$ with $f(\beta)\leq f(\alpha)$ and  at most one $\gamma^{(n-1)}<\alpha^{(n)}$ with $f(\gamma)\geq f(\alpha)$.   Then he defined discrete gradient vector fields $\text{grad } f$ by assigning arrows from $\alpha$ to $\beta$ whenever $\alpha^{(n)}<\beta^{(n+1)}$ %for some $n\geq 0$ 
and $f(\alpha)\geq f(\beta)$.   The critical simplices are essentially the simplices which are neither the heads nor the tails of any arrows in $\text{grad } f$. % essentially determined by $\text{grad }f$. 
R. Forman constructed a chain complex consisting of the formal linear combinations of the critical simplices.  He proved that the homology of the new chain complex is isomorphic to the homology of the original simplicial complex.  % He gave expressions of the cup-product and the cohomology operation by using the critical simplices.  He  also gave a Witten-Morse theory on simplicial complexes, based on the critical simplices. 

%Some algorithms to define discrete gradient vector fields on simplicial complexes (cell complexes) are given by T. Lewiner, H. Lopes and G. Tavares \cite{comp}. 

\smallskip

In this paper,  we generalize the discrete Morse theory for simplicial complexes in  \cite{ forman1,forman2,forman3} and give a discrete Morse theory for hypergraphs. 
Let  $\mathcal{K}$ be a simplicial complex such that $\mathcal{H}\subseteq\mathcal{K}$.  %contains all the hyperedges of $\mathcal{H}$ as simplices.   
% for simplicial complexes. 
Let $\overline f_{\mathcal{K}}$ be a discrete Morse function on $\mathcal{K}$.  We denote the collection of all the critical simplices of $\overline f_\mathcal{K}$ as  $M(\overline f_{\mathcal{K}},\mathcal{K})$.   The first main result   is the next theorem.

\begin{theorem}[Main Result I]%[Theorem~\ref{th-6.5}]
\label{th-0.05}
Let $\mathcal{H}$ be a hypergraph and $n\geq 0$. {\color{black} Suppose both $\text{Inf}_*(\mathcal{H})$ and $\text{Sup}_*(\mathcal{H})$ are $\text{grad}~f$-invariant. } Then we have 
the following isomorphisms of homology groups
\begin{eqnarray*}
H_n(\mathcal{H};R)&\cong& H_n(\{R(M(\overline f_\mathcal{K}, \mathcal{K}))_k\cap \text{Inf}_k(R(\mathcal{H})_*),\tilde\partial^{\mathcal{K}}_k\}_{k\geq 0})\\
&\cong&  H_n(\{R(M(\overline f_\mathcal{K}, \mathcal{K}))_k\cap \text{Sup}_k(R(\mathcal{H})_*),\tilde\partial^{\mathcal{K}}_k\}_{k\geq 0}). 
\end{eqnarray*}
Here 
$
\tilde \partial^{\mathcal{K}}_k%=(\pi_{M}\mid_{C_n^{\overline\Phi}(\Delta\mathcal{H};R)})\circ \partial_k\circ  (\overline\Phi^\infty\mid _{R(M(\overline f,\Delta\mathcal{H}))_*})
$
 is (the restriction of)  
 the boundary map from $R(M(\overline f_\mathcal{K}, \mathcal{K}))_k$ to $R(M(\overline f_\mathcal{K}, \mathcal{K}))_{k-1}$ whose  explicit formula  is given in \cite[Theorem~8.10]{forman1}.   
\end{theorem}

We also  generalize discrete  Morse functions and discrete gradient vector fields on simplicial complexes and define them on hypergraphs.  We define discrete Morse functions $g$ on $\mathcal{H}$ by assigning a real number $g(\sigma)$ to each hyperedge $\sigma\in\mathcal{H}$ such that  for any $n\geq 0$ and any hyperedge $\alpha^{(n)}$, there exist at most one hyperedge $\beta^{(n+1)}>\alpha^{(n)}$ with $g(\beta)\leq g(\alpha)$ and  at most one hyperedge $\gamma^{(n-1)}<\alpha^{(n)}$ with $g(\gamma)\geq g(\alpha)$.   Then we define discrete gradient vector fields $\text{grad } g$ on $\mathcal{H}$ by assigning arrows from a hyperedge $\alpha$ to a hyperedge $\beta$ whenever $\alpha^{(n)}<\beta^{(n+1)}$ for some $n\geq 0$ and $g(\alpha)\geq g(\beta)$.   The critical hyperedges are essentially the hyperedges which are neither  the heads nor the tails  of any arrows in $\text{grad } g$. 
 We denote the collection of all the critical hyperedges of $g$ as  $M(g,\mathcal{H})$. 
% Let $\mathcal{K}$ be a simplicial complex containing all the hyperedges . 
The second main result is the next theorem. %follows from Theorem~\ref{th-0.05}.
    
    \begin{theorem}[Main Result II]%[Corollary~\ref{co-6.888}]
    \label{co-0.08}
    Let $\mathcal{H}$ be  a hypergraph. Suppose for any $k\geq 1$ and any hyperedges $\beta^{(k+1)}>\alpha^{(k)}>\gamma^{(k-1)}$ of $\mathcal{H}$, there exists $\hat   \alpha^{(k)}\in\mathcal{H}$, $\hat   \alpha\neq \alpha$,  such that $\beta>\hat  \alpha>\gamma$.  Let $g$ be a discrete Morse function on  $\mathcal{H}$. Then there exists a discrete Morse function $\overline f_\mathcal{K}$ on $\mathcal{K}$ such that $\text{grad } (\overline f_\mathcal{K})$ is $\text{grad } g$ on $\mathcal{H}$, and $\text{grad } (\overline f_\mathcal{K})$ is vanishing on $\mathcal{K}\setminus\mathcal{H}$.  Moreover,  if {\color{black}  both $\text{Inf}_*(\mathcal{H})$ and $\text{Sup}_*(\mathcal{H})$ are $\text{grad}~g$-invariant, } 
\begin{eqnarray}\label{eq-0.0a}
H_n(\mathcal{H};R)\cong H_n(\{R(M(g,\mathcal{H}))_k\cap (\partial^{\mathcal{K}}_k)^{-1}(R(\mathcal{H})_{k-1}),\tilde\partial^{\mathcal{K}}_k\}_{k\geq 0}). 
\end{eqnarray}
 Here 
$
\tilde \partial^{\mathcal{K}}_k%=(\pi_{M}\mid_{C_n^{\overline\Phi}(\Delta\mathcal{H};R)})\circ \partial_k\circ  (\overline\Phi^\infty\mid _{R(M(\overline f,\Delta\mathcal{H}))_*})
$
 is (the restriction of) the boundary map from $R(M(\overline f_\mathcal{K}, \mathcal{K}))_k$ to $R(M(\overline f_\mathcal{K}, \mathcal{K}))_{k-1}$ whose  explicit formula  is given in \cite[Theorem~8.10]{forman1}.  Furthermore,   the chain complex 
  \begin{eqnarray}\label{eq-0.i}
 \{R(M(g,\mathcal{H}))_k\cap (\partial^{\mathcal{K}}_k)^{-1}(R(\mathcal{H})_{k-1}),\tilde\partial^{\mathcal{K}}_k\}_{k\geq 0}
 \end{eqnarray}
  is determined by $\mathcal{H}$ and $g$, and does not depend on the choice of $\mathcal{K}$ and the choice of $\overline f_\mathcal{K}$.    
        \end{theorem}

    \smallskip
    
    The paper is organized as follows. 
In Section~\ref{s2}, we give some preliminary knowledge of the associated simplicial complexes and the  embedded homology of hypergraphs.   In  Section~\ref{s3}$-$\ref{s6},  we develop a discrete Morse theory for hypergraphs by studying the discrete Morse functions, discrete gradient vector fields,   discrete gradient flows,  and the critical simplices and critical hyperedges subsequently. Then  in Section~\ref{s7},  we prove Theorem~\ref{th-0.05} and Theorem~\ref{co-0.08}.

  We give some by-products in Section~\ref{s8}$-$\ref{s10}. % of Theorem~\ref{th-0.05},  
 We prove some Morse inequalities for hypergraphs in Theorem~\ref{th-77.1} and Theorem~\ref{th-77.2}.  We  prove that the collapses of hypergraphs preserve  the embedded homology,  in Theorem~\ref{pr-8.4}.  We  study the collapse of level hypergraphs in Corollary~\ref{co-8.8}.

  We point out that an alternative approach to construct a discrete Morse theory for hypergraphs   is to apply the algebraic Morse theory in \cite{mams,alg2,algm} to the chain complexes (\ref{eq-0.01}) and (\ref{eq-0.02}).  
Nevertheless,  our approach has the advantage that we can observe the geometric meaning of the critical hyperedges, hence we can   apply our discrete Morse theory directly to concrete hypergraphs. % collapses and level hypergraphs. 

  \smallskip
  
 Finally, we discuss some potential research which is not included in the remaining sections of this paper. %our discrete Morse theory for hypergraphs.  %The potential ideas in the following discussions have not been developed in  this paper. 
%In  \cite{mams,alg2,algm},  
 % the discrete Morse theory by Forman \cite{ forman1,forman2,forman3} has been generalized  to   an algebraic Morse theory for chain complexes.  
(i).  We intend to study the potential connections between the discrete Morse theory of graphs \cite{ayala2,ayala1} and our discrete Morse theory of hypergraphs.  
(ii).  Inspired by \cite{cy}, we intend to equip a weight on each hyperedge and consider weighted discrete Morse theory for weighted hypergraphs.  %Moreover,  inspired by \cite{cy},  weights could be equipped on hypergraphs as welthe boundary maps of (\ref{eq-0.01}) and (\ref{eq-0.02}). %   and  a discrete Morse theory could be constructed for the weighted embedded  homology. %In order to study the topological aspect of the discrete Morse theory for hypergraphs, we st  %In order to study the topological proper % However, since the algebraic Morse theory in \cite{mams,algm} is abstracted from Forman's discrete Morse theory for simplicial complexes, we expect that we would get Theorem~\ref{th-0.05} as well. 
   % \item
   % In \cite{forman3},  the cup-product of the cohomology of simplicial complexes was expressed in terms of the critical simplices.  By applying \cite{forman3} to $\mathcal{K}$ in   Theorem~\ref{th-0.05},  we will potentially  get the cup-product of $\mathcal{H}$. % of hypergraphs is not %,  the relations between the cup-product of $\mathcal{K}$ in terms of the critical simplices and  the embedded homology of $\mathcal{H}$. %, in the future.  
   % \item
    %In \cite{forman8},    R. Forman gave a Witten-Morse theory for simplicial complexes.  By applying \cite{forman8} to $\mathcal{K}$ in   Theorem~\ref{th-0.05},  we will potentially get the Witten-Morse theory for $\mathcal{H}$. %, in the future. 
  
  % And in \cite{pers},  K. Mischaikow and V. Nanda applied   discrete Morse theory to give an efficient preprocessing algorithm which reduces the number of simplices (cells) in a filtered simplicial complex (cell complex)  and   computes the persistent homology.   As a potential generalization of \cite{pers}, we want to  apply our discrete Morse theory of hypergraphs to reduce the number of hyperedges and compute the persistent embedded homology. % of hypergraphs. %, in the future. 

    \section{Associated Simplicial Complexes  and Embedded Homology }\label{s2}
    
  Let $\mathcal{H}$ be a hypergraph.  In this section, we introduce the associated simplicial complex $\Delta\mathcal{H}$, the lower-associated simplicial complex $\delta\mathcal{H}$, and the embedded homology $H_*(\mathcal{H})$. 
  
\smallskip
 
 {\bf 1.  Associated Simplicial Complexes}.  Firstly,  we introduce $\Delta\mathcal{H}$ and $\delta\mathcal{H}$.  
 For a single hyperedge $\sigma=\{v_0,v_1,\ldots,v_n\}$ of $\mathcal{H}$, the {\it associated simplicial complex} $\Delta\sigma$ of $\sigma$ is   the collection of all the nonempty subsets of $\sigma$
\begin{eqnarray}\label{eq-md1}
\Delta\sigma=\{\{v_{i_0},v_{i_1},\ldots,v_{i_k}\}\mid 0\leq i_0<i_1<\cdots< i_k\leq n,  0\leq k\leq n\}. 
\end{eqnarray}
    The {\it associated simplicial complex} $\Delta \mathcal{H}$ of $\mathcal{H}$ is the smallest simplicial complex that $\mathcal{H}$ can be embedded in (cf. \cite{parks}). Explicitly,  $\Delta\mathcal{H}$ has its set of simplices as the union of the $\Delta\sigma$'s for all $\sigma\in\mathcal{H}$
\begin{eqnarray}\label{e10}
\Delta \mathcal{H}=\{\tau\in\Delta\sigma\mid \sigma\in \mathcal{H}% \text{ and  } \eta\neq\emptyset
\}.
\end{eqnarray} 
We use 
\begin{eqnarray*}
\partial_*: C_*(\Delta\mathcal{H};R)\longrightarrow C_{*-1}(\Delta\mathcal{H};R), \text{\ \ } *=0,1,2,\ldots,
\end{eqnarray*}
 to denote the boundary maps of $\Delta\mathcal{H}$.  For simplicity, we sometimes denote $\partial_*$ as $\partial$ and omit the dimensions.  
Let the {\it lower-associated simplicial complex} $\delta \mathcal{H}$ be the largest simplicial complex that can be embedded in $\mathcal{H}$. Then the set of simplices of $\delta\mathcal{H}$ consists of the hyperedges  $\sigma\in\mathcal{H}$ whose associated simplicial complexes $\Delta\sigma$ are subsets of $\mathcal{H}$ 
\begin{eqnarray}\label{e11}
\delta\mathcal{H}&=&\{\sigma\in\mathcal{H}\mid \Delta\sigma\subseteq \mathcal{H}\}\nonumber\\
&=&\{\tau\in\Delta\sigma\mid \Delta\sigma\subseteq \mathcal{H}\}. 
\end{eqnarray}
We notice that as hypergraphs,  
\begin{eqnarray}\label{e1}
\delta \mathcal{H}\subseteq \mathcal{H}\subseteq\Delta\mathcal{H}.
\end{eqnarray}
By (\ref{e10}) and (\ref{e11}), it can be verified that in (\ref{e1}), one of the equalities holds iff. both of the equalities hold iff. $\mathcal{H}$ is a simplicial complex.  
 The boundary map of $\delta\mathcal{H}$  is the restriction of $\partial_*$ to $\delta\mathcal{H}$
\begin{eqnarray*}
\partial_*\mid_{C_*(\delta\mathcal{H};R)}: C_*(\delta\mathcal{H};R)\longrightarrow C_{*-1}(\delta\mathcal{H};R), \text{\ \ }*=0,1,2,\ldots. 
\end{eqnarray*}

\smallskip

{\bf 2. Embedded Homology}.  Secondly,  we introduce the embedded homology $H_*(\mathcal{H})$.  Let $D_*$ be a graded sub-$R$-module of $C_*(\Delta\mathcal{H};R)$.  By \cite[Section~2]{h1}, the infimum chain complex is  defined as
\begin{eqnarray*}
\text{Inf}_n(D_*)= D_n\cap\partial_n^{-1}(D_{n-1}),  \text{\ \ \ } n\geq 0,
\end{eqnarray*}
which   is  the largest sub-chain complex of $C_*(\Delta\mathcal{H};R)$ contained in $D_*$ as graded $R$-modules;   and the supremum chain complex is defined as
\begin{eqnarray*}
\text{Sup}_n(D_*)= D_n+\partial_{n+1}(D_{n+1}), \text{\ \ \ }n\geq 0,
\end{eqnarray*}
which is the smallest sub-chain complex of $C_*(\Delta\mathcal{H};R)$ containing $D_*$ as graded $R$-modules.  Moreover, 
 as chain complexes,    
\begin{eqnarray}
\{\text{Inf}_n(D_*), \partial_n\mid_{\text{Inf}_n(D_*)}\}_{n\geq 0}&\subseteq& \{\text{Sup}_n(D_*), \partial_n\mid_{\text{Sup}_n(D_*)}\}_{n\geq 0} \nonumber\\
&\subseteq& \{C_n(\Delta\mathcal{H};R), \partial_n\}_{n\geq 0}. 
\label{eq-2.99}
\end{eqnarray}
By \cite[Section~2]{h1}, the canonical inclusion 
\begin{eqnarray*}
\iota: \text{Inf}_n(D_*)\longrightarrow \text{Sup}_n(D_*), \text{\ \ \ }n\geq 0, 
\end{eqnarray*}
induces an isomorphism of homology groups
\begin{eqnarray}\label{eq-0.1}
\iota_*: H_*(\{\text{Inf}_n(D_*), \partial_n\mid_{\text{Inf}_n(D_*)}\}_{n\geq 0})\overset{\cong}{\longrightarrow}H_*(\{\text{Sup}_n(D_*), \partial_n\mid_{\text{Sup}_n(D_*)}\}_{n\geq 0}). 
\end{eqnarray}
We call the homology groups in (\ref{eq-0.1}) the {\it embedded homology} of $D_*$ (cf. \cite[Section~2]{h1}), and denote the homology groups as $H_n(D_*)$, $n\geq 0$.

Now   we take  
\begin{eqnarray}\label{eq-2.99999}
D_n= R(\mathcal{H})_n, \text{\ \ \ }n\geq 0.  
\end{eqnarray} 
As chain complexes,   we have the inclusions  
\begin{eqnarray*}
&~&\{C_n(\delta\mathcal{H};R), \partial_n\mid_{C_n(\delta\mathcal{H};R)}\}_{n\geq 0}\subseteq\{\text{Inf}_n(R(\mathcal{H})_*), \partial_n\mid_{\text{Inf}_n(R(\mathcal{H})_*)}\}_{n\geq 0}\\
&\subseteq& \{\text{Sup}_n(R(\mathcal{H})_*), \partial_n\mid_{\text{Sup}_n(R(\mathcal{H})_*)}\}_{n\geq 0} % \\
\subseteq \{C_n(\Delta\mathcal{H};R), \partial_n\}_{n\geq 0}. 
\end{eqnarray*}
By substituting (\ref{eq-2.99999}) into (\ref{eq-0.1}), the {\it embedded homology groups} of $\mathcal{H}$ is defined as (cf. \cite[Subsection~3.2]{h1})
\begin{eqnarray*}
H_n(\mathcal{H})= H_n(R(\mathcal{H})_*), \text{\ \ \ }n\geq 0. 
\end{eqnarray*}
 In particular, if $\mathcal{H}$ is a simplcial complex, then (\ref{eq-2.99999}) is a chain complex and 
\begin{eqnarray*}
R(\mathcal{H})_n=\text{Inf}_n(R(\mathcal{H})_*)=\text{Sup}_*(R(\mathcal{H})_*), \text{\ \ \ }n\geq 0.  
\end{eqnarray*}
The embedded homology is reduced to the usual homology of simplcial complexes.  

\smallskip

{\bf 3. Morphisms of Hypergraphs}.  Thirdly,  we observe that morphisms between hypergraphs induce simplicial maps between the (lower-)associated simplicial complexes and homomorphisms between the embedded homology.  Let $\mathcal{H}$ and $\mathcal{H}'$ be hypergraphs. A {\it morphism} of hypergraphs from $\mathcal{H}$ to $\mathcal{H}'$ is  a map $\varphi$ sending a vertex of $\mathcal{H}$ to a vertex of $\mathcal{H}'$ such that whenever $\sigma=\{v_0,\ldots,v_k\}$ is a hyperedge of $\mathcal{H}$, $\varphi(\sigma) = \{\varphi(v_0), \ldots, \varphi(v_k)\}$ is a hyperedge of $\mathcal{H}'$.  Let $\varphi:\mathcal{H}\longrightarrow\mathcal{H}'$ be such a morphism.  By an  argument similar to \cite[Section~3.1]{h1},  $\varphi$  induces two simplicial maps 
\begin{eqnarray*}
\delta\varphi:  \delta \mathcal{H}\longrightarrow \delta\mathcal{H}',~~~~~~
\Delta\varphi:  \Delta \mathcal{H}\longrightarrow \Delta\mathcal{H}'
\end{eqnarray*}
such that $\varphi=(\Delta\varphi)\mid_{\mathcal{H}}$ and $\delta\varphi=\varphi\mid_{\delta\mathcal{H}}$.  
The simplicial maps $\delta\varphi$ and $\Delta\varphi$   induce  homomorphisms between the homology groups 
\begin{eqnarray}
(\delta\varphi)_*:&& H_*(\delta\mathcal{H})\longrightarrow H_*(\delta\mathcal{H}'),\label{e88-1}\\
(\Delta\varphi)_*:&& H_*(\Delta\mathcal{H})\longrightarrow H_*(\Delta\mathcal{H}'). \label{e88-2}
\end{eqnarray}
Moreover, by \cite[Proposition~3.7]{h1}, we have an induced homomorphism between the embedded homology 
\begin{eqnarray}
\varphi_*:&& H_*(\mathcal{H})\longrightarrow H_*(\mathcal{H}')\label{e88-3}.
\end{eqnarray}
In particular, if both $\mathcal{H}$ and $\mathcal{H}'$ are simplicial complexes, then $\varphi$ is a simplicial map     and  the three homomorphisms $(\delta\varphi)_*$, $(\Delta\varphi)_*$ and $\varphi_*$ are the same. 
 %by applying an analogous argument of \cite[Proposition~3.7]{h1} to

\smallskip

{\bf 4. Examples}. Finally,  we give some examples to show that  $\Delta\mathcal{H}$,  $\delta\mathcal{H}$ and $H_*(\mathcal{H})$ detect  $\mathcal{H}$  from different aspects, in the remaining part of this section.

% Given two hypergraphs $\mathcal{H}$ and $\mathcal{H}'$, we consider the following  conditions:
%(1). 
% $\delta \mathcal{H}= \delta \mathcal{H}'$;
%(2). 
%  $\Delta \mathcal{H}= \Delta \mathcal{H}'$;
%(3). 
%  $H_*( \mathcal{H})\cong H_*(\mathcal{H}')$.
% The next example shows that (1) and (2) cannot imply (3). 

\begin{example}\label{ex82}
Let 
\begin{eqnarray*}
\mathcal{H}&=&\{ \{v_0,v_1,v_2,v_3\},\{v_0\}\},\\
\mathcal{H}'&=&\{ \{v_0,v_1,v_2,v_3\}, \{v_0,v_1\}, \{v_0,v_2\},\{v_0,v_3\},\{v_1,v_2\},\{v_1,v_3\},\{v_2,v_3\},\{v_0\}\}. 
\end{eqnarray*}
  Then  $\delta\mathcal{H}=\delta\mathcal{H}'=\{\{v_0\}\}$, and both $\Delta\mathcal{H}$ and $\Delta\mathcal{H}'$ are the tetrahedron.   Moreover,
$H_1(\mathcal{H})=0$ and $H_1(\mathcal{H}')=\mathbb{Z}^{\oplus 3}$. 
\end{example}

%\noindent The next example shows that (1) and (3) cannot imply (2).  
 
\begin{example}\label{ex81}
Let 
\begin{eqnarray*}
\mathcal{H}&=&\{\{v_0\},\{v_1\},\{v_2\}, \{v_3\},\{v_4\},\{v_5\},\\
&&\{v_0,v_1,v_3\}, \{v_1,v_2,v_4\},\{v_3,v_4,v_5\}\},\\
\mathcal{H}'&=&\{\{v_0\},\{v_1\},\{v_2\}, \{v_3\},\{v_4\},\{v_5\},\\
&&\{v_0,v_1,v_3\}, \{v_1,v_2,v_4\},\{v_1,v_3,v_4\},\{v_3,v_4,v_5\}\}. 
\end{eqnarray*}
  Then  $\delta\mathcal{H}=\delta\mathcal{H}'=\{\{v_0\},\{v_1\},\{v_2\}, \{v_3\},\{v_4\},\{v_5\}\}$, 
$H_n(\mathcal{H})=H_n(\mathcal{H}')=0$ for $n\geq 1$, and $H_0(\mathcal{H})=H_0(\mathcal{H}')=\mathbb{Z}^{\oplus 6}$.  
Moreover, $H_1(\Delta\mathcal{H})=\mathbb{Z}$, and $H_1(\Delta{\mathcal{H}'})=0$.

\begin{figure}[!htbp]
 \begin{center}
\begin{tikzpicture}
\coordinate [label=below left:$v_0$]    (A) at (1,0); 
 \coordinate [label=below right:$v_1$]   (B) at (2.5,0); 
 \coordinate  [label=below right:$v_2$]   (C) at (4,0); 
\coordinate  [label=left:$v_3$]   (D) at (1.75,2/2); 
\coordinate  [label=right:$v_4$]   (E) at (6.5/2,2/2); 
\coordinate  [label=right:$v_5$]   (F) at (5/2,4/2);

\fill (1,0) circle (2.5pt);
\fill (2.5,0) circle (2.5pt);
\fill (4,0) circle (2.5pt);
\fill (1.75,2/2) circle (2.5pt); 
\fill (6.5/2,2/2) circle (2.5pt); 
\fill (5/2,4/2) circle (2.5pt);

 \coordinate[label=left:$\mathcal{H}$:] (G) at (0.5/2,2/2);
 \draw [dashed,thick] (A) -- (B);
 \draw [dashed,thick] (B) -- (C);
 %\draw [dashed,thick] (C) -- (A);
  \draw [dashed,thick] (D) -- (A);
\draw [dashed,thick] (D) -- (B);
\draw [dashed,thick] (E) -- (C);
\draw [dashed,thick] (E) -- (B);
\draw [dashed,thick] (E) -- (F);
\draw [dashed,thick] (D) -- (F);
\draw [dashed, thick] (D) -- (E);

\fill [fill opacity=0.25][gray!100!white] (A) -- (B) -- (D) -- cycle;
\fill [fill opacity=0.25][gray!100!white] (B) -- (C) -- (E) -- cycle;
\fill [fill opacity=0.25][gray!100!white] (D) -- (E) -- (F) -- cycle;

\coordinate [label=below left:$v_0$]    (A) at (1+6,0); 
 \coordinate [label=below right:$v_1$]   (B) at (2.5+6,0); 
 \coordinate  [label=below right:$v_2$]   (C) at (4+6,0); 
\coordinate  [label=left:$v_3$]   (D) at (1.75+6,2/2); 
\coordinate  [label=right:$v_4$]   (E) at (6.5/2+6,2/2); 
\coordinate  [label=right:$v_5$]   (F) at (5/2+6,4/2); 

 \coordinate[label=left:$\Delta{\mathcal{H}}$:] (G) at (0.5/2+6,2/2);
 \draw  [thick](A) -- (B);
 \draw [thick] (B) -- (C);
 %\draw [dashed,thick] (C) -- (A);
  \draw [thick] (D) -- (A);
\draw [thick] (D) -- (B);
\draw [thick] (E) -- (C);
\draw [thick] (E) -- (B);
\draw [thick] (E) -- (F);
\draw [thick] (D) -- (F);
\draw [thick] (D) -- (E);

\fill [fill opacity=0.25][gray!100!white] (A) -- (B) -- (D) -- cycle;
\fill [fill opacity=0.25][gray!100!white] (B) -- (C) -- (E) -- cycle;
\fill [fill opacity=0.25][gray!100!white] (D) -- (E) -- (F) -- cycle;

\fill (7,0) circle (2.5pt);
\fill (8.5,0) circle (2.5pt);
\fill (10,0) circle (2.5pt);
\fill (7.75,2/2) circle (2.5pt); 
\fill (6.5/2+6,2/2) circle (2.5pt); 
\fill (5/2+6,4/2) circle (2.5pt); 

 \end{tikzpicture}
\end{center}
%\end{example}
%\begin{example}\label{ex82}

 \begin{center}
\begin{tikzpicture}
\coordinate [label=below left:$v_0$]    (A) at (2/2,0); 
 \coordinate [label=below right:$v_1$]   (B) at (5/2,0); 
 \coordinate  [label=below right:$v_2$]   (C) at (8/2,0); 
\coordinate  [label=left:$v_3$]   (D) at (3.5/2,2/2); 
\coordinate  [label=right:$v_4$]   (E) at (6.5/2,2/2); 
\coordinate  [label=right:$v_5$]   (F) at (5/2,4/2);

\fill (1,0) circle (2.5pt);
\fill (2.5,0) circle (2.5pt);
\fill (4,0) circle (2.5pt);
\fill (1.75,2/2) circle (2.5pt); 
\fill (6.5/2,2/2) circle (2.5pt); 
\fill (5/2,4/2) circle (2.5pt);

 \coordinate[label=left:$\mathcal{H}'$:] (G) at (0.5/2,2/2);
 \draw [dashed,thick] (A) -- (B);
 \draw [dashed,thick] (B) -- (C);
 %\draw [dashed,thick] (C) -- (A);
  \draw [dashed,thick] (D) -- (A);
\draw [dashed,thick] (D) -- (B);
\draw [dashed,thick] (E) -- (C);
\draw [dashed,thick] (E) -- (B);
\draw [dashed,thick] (E) -- (F);
\draw [dashed,thick] (D) -- (F);
\draw [dashed,thick] (D) -- (E);

\fill [fill opacity=0.25][gray!100!white] (A) -- (B) -- (D) -- cycle;
\fill [fill opacity=0.25][gray!100!white] (B) -- (C) -- (E) -- cycle;
\fill [fill opacity=0.25][gray!100!white] (D) -- (E) -- (F) -- cycle;
\fill [fill opacity=0.25][gray!100!white] (D) -- (E) -- (B) -- cycle;

\coordinate [label=below left:$v_0$]    (A) at (1+6,0); 
 \coordinate [label=below right:$v_1$]   (B) at (2.5+6,0); 
 \coordinate  [label=below right:$v_2$]   (C) at (4+6,0); 
\coordinate  [label=left:$v_3$]   (D) at (1.75+6,2/2); 
\coordinate  [label=right:$v_4$]   (E) at (6.5/2+6,2/2); 
\coordinate  [label=right:$v_5$]   (F) at (5/2+6,4/2); 

 \coordinate[label=left:$\Delta{\mathcal{H}'}$:] (G) at (0.5/2+6,2/2);
 \draw [thick] (A) -- (B);
 \draw [thick] (B) -- (C);
 %\draw [dashed,thick] (C) -- (A);
  \draw [thick] (D) -- (A);
\draw [thick] (D) -- (B);
\draw [thick] (E) -- (C);
\draw [thick] (E) -- (B);
\draw [thick] (E) -- (F);
\draw [thick] (D) -- (F);
\draw [thick] (D) -- (E);

\fill [fill opacity=0.25][gray!100!white] (A) -- (B) -- (D) -- cycle;
\fill [fill opacity=0.25][gray!100!white] (B) -- (C) -- (E) -- cycle;
\fill [fill opacity=0.25][gray!100!white] (D) -- (E) -- (F) -- cycle;
\fill [fill opacity=0.25][gray!100!white] (D) -- (E) -- (B) -- cycle;

\fill (7,0) circle (2.5pt);
\fill (8.5,0) circle (2.5pt);
\fill (10,0) circle (2.5pt);
\fill (7.75,2/2) circle (2.5pt); 
\fill (6.5/2+6,2/2) circle (2.5pt); 
\fill (5/2+6,4/2) circle (2.5pt); 
 \end{tikzpicture}
\end{center}
\caption{Example~\ref{ex81}.}
\end{figure}
\end{example}

Consider a hypergraph morphism $\varphi: \mathcal{H}\longrightarrow\mathcal{H}'$.  
The following example shows that (\ref{e88-1}), (\ref{e88-2}) and (\ref{e88-3})   can be distinct. 

\begin{example}\label{ex1}
Let $\mathcal {H}=\{\{v_0,v_1\},\{v_1,v_2\},\{v_0,v_2\}\}$. Let $\sigma=\{v_0,v_1,v_2\}$. Let $\mathcal{H}'=\mathcal{H}\bigcup\{\sigma\}$.   Let $\varphi$ be the canonical inclusion of $\mathcal{H}$ into $\mathcal{H}'$. Then 
\begin{enumerate}[(a). ]
\item
$H_0(\mathcal{H})=H_0(\mathcal{H}')=0$, $H_1(\mathcal{H})=\mathbb{Z}$, and $H_1(\mathcal{H}')=0$. Moreover,  $\varphi_*$ is the identity map from zero to zero in dimension $0$, and the zero map on $\mathbb{Z}$ in dimension $1$;
\item
$\delta \mathcal{H}=\delta\mathcal{H}'=\emptyset$. Moreover, $(\delta\varphi)_*$ is the identity map from zero to zero in all dimensions;
\item
$\Delta \mathcal{H}\simeq S^1$ and $\Delta\mathcal{H}'\simeq *$. Moreover, $(\Delta\varphi)_*$ is the identity map on $\mathbb{Z}$ in dimension $0$, and the zero map on $\mathbb{Z}$ in dimension $1$.   
\end{enumerate} 
\end{example}

%\begin{remark}\label{r88}
%\end{remark}

%We notice that $R(\delta\mathcal{H})_*$ is a chain complex  whose boundary map is the restriction of  $\partial_*$ to $\mathbb{R}(\delta \mathcal{H})_*$.  Moreover,  as graded $R$-modules,
%\begin{eqnarray*}
%\mathbb{R}(\delta\mathcal{H})_*\subseteq \mathbb{R}(\mathcal{H})_*\subseteq \mathbb{R}(\Delta\mathcal{H})_*.
% \end{eqnarray*}

    \section{Discrete Morse Functions  and Critical Hyperedges}\label{s3}
    
        In this section, we define the discrete  Morse functions on $\mathcal{H}$ and the critical hyperedges of the discrete Morse functions. 
    
    \smallskip
    
%\subsection{ Discrete Morse functions on hypergraphs }
    
    %Firstly,  we define discrete Morse functions. 

    \begin{definition}\label{def1}
   A function $f: \mathcal{H}\longrightarrow \mathbb{R}$ is called a {\it discrete Morse function} on $\mathcal{H}$ if for every $n\geq 0$ and every  $\alpha^{(n)}\in \mathcal{H}$,  both of the two conditions hold:
    \begin{enumerate}[(i).]
    \item
   % \begin{eqnarray*}
    $\#\{\beta^{(n+1)}>\alpha^{(n)}\mid f(\beta)\leq f(\alpha), \beta\in \mathcal{H}\}\leq 1$;
   % \end{eqnarray*}
    \item
     %  \begin{eqnarray*}
    $\#\{\gamma^{(n-1)}<\alpha^{(n)}\mid f(\gamma)\geq f(\alpha), \gamma\in \mathcal{H}\}\leq 1$. 
   % \end{eqnarray*}
    \end{enumerate}
    \end{definition}
    
    Let $f$ be a discrete Morse function on $\mathcal{H}$. 

\begin{definition}\label{def2}
A hyperedge $\alpha^{(n)}\in\mathcal{H}$ is called {\it critical} if  
both of the following two conditions hold:
\begin{enumerate}[(i).]
\item
%\begin{eqnarray*}
$\#\{\beta^{(n+1)}>\alpha^{(n)}\mid f(\beta)\leq f(\alpha),\beta\in\mathcal{H}\}=0$;
%\end{eqnarray*}
\item
%\begin{eqnarray*}
$\#\{\gamma^{(n-1)}<\alpha^{(n)}\mid f(\gamma)\geq f(\alpha),\gamma\in\mathcal{H}\}=0$. 
%\end{eqnarray*}
\end{enumerate}
\end{definition}
We use $M(f,\mathcal{H})$ to denote the set of all critical hyperedges.   For $\alpha^{(n)}\in\mathcal{H}$,  it is direct to see that $\alpha\notin M(f,\mathcal{H})$  if at least one of the following conditions hold:

\begin{enumerate}[(A).]
\item
there exists $\beta^{(n+1)}>\alpha^{(n)}$, $\beta\in\mathcal{H}$, such that $f(\beta)\leq f(\alpha)$;
\item
there exists $\gamma^{(n-1)}<\alpha^{(n)}$, $\gamma\in\mathcal{H}$, such that $f(\gamma)\geq f(\alpha)$. 
\end{enumerate}
    
    In particular, if $\mathcal{H}$ is a simplicial complex, then 
    the discrete Morse functions defined in Definition~\ref{def1} are the same as the discrete Morse functions defined in \cite[Definition~2.1]{forman1};   
and the critical hyperedges defined in Definition~\ref{def2} are the same as the critical simplices defined in \cite[Definition~2.2]{forman1}.

     The next lemma follows from Definition~\ref{def1}. 
 It is a generalization of  \cite[Lemma~2.1]{forman1}.  %The lemma 

\begin{lemma}\label{lem1}
Let $\mathcal{H}$ and $\mathcal{H}'$ be two hypergraphs such that $\mathcal{H}'\subseteq \mathcal{H}$. Suppose $f:\mathcal{H}\longrightarrow\mathbb{R}$ is a discrete Morse function on $\mathcal{H}$. Let $f'=f\mid_{\mathcal{H}'}$  be  the restriction of $f$ to $\mathcal{H}'$.  Then $f'$ is a discrete Morse function on $\mathcal{H}'$.  \qed
\end{lemma}
%\begin{proof}
%For every $\alpha^{(n)}\in \mathcal{H}$, since $\mathcal{H}'\subseteq \mathcal{H}$, 
%\begin{eqnarray*}
%\{\beta^{(n+1)}>\alpha^{(n)} \mid f'(\beta)\leq f'(\alpha), \beta\in\mathcal{H}'\}\subseteq \{\beta^{(n+1)}>\alpha^{(n)} \mid f(\beta)\leq f(\alpha), \beta\in\mathcal{H}\}. 
%\end{eqnarray*}
%Since $f$ is a discrete Morse function on $\mathcal{H}$,
%\begin{eqnarray*}
%\#\{\beta^{(n+1)}>\alpha^{(n)} \mid f'(\beta)\leq f'(\alpha), \beta\in\mathcal{H}'\}\leq \# \{\beta^{(n+1)}>\alpha^{(n)} \mid f(\beta)\leq f(\alpha), \beta\in\mathcal{H}\}\leq 1. 
%\end{eqnarray*}
%Similarly, 
%\begin{eqnarray*}
%\#\{\gamma^{(n-1)}<\alpha^{(n)} \mid f'(\gamma)\geq f'(\alpha), \gamma\in\mathcal{H}'\}\leq %\# \{\gamma^{(n-1)}<\alpha^{(n)} \mid f(\gamma)\geq f(\alpha), \gamma\in\mathcal{H}\}\leq 
%1. 
%\end{eqnarray*}
%Hence $f'$ is a discrete Morse function on $\mathcal{H}'$. 
%\end{proof}

%In particular, if $\mathcal{H}$ is a simplicial complex, then 
%Lemma~\ref{lem1} is reduced to

The next proposition follows from Lemma~\ref{lem1} immediately. 

\begin{proposition}\label{pr1}
\begin{enumerate}[(i).]
\item
Let $\overline f: \Delta\mathcal{H}\longrightarrow\mathbb{R}$ be a discrete Morse function on $\Delta\mathcal{H}$. Then $f=\overline f\mid_{\mathcal{H}}$ is a discrete Morse function on $\mathcal{H}$;
\item
Let $f:\mathcal{H}\longrightarrow \mathbb{R}$ be a discrete Morse function on $\mathcal{H}$. Then ${\underline   f}=f\mid _{\delta\mathcal{H}}$ is a discrete Morse function on $\delta\mathcal{H}$. \qed 
\end{enumerate}
\end{proposition}

 By \cite[Section~4]{forman1},  there always exists a discrete Morse function on $\Delta\mathcal{H}$.  Hence the next corollary is a consequence of \cite{forman1} and Proposition~\ref{pr1}. 

\begin{corollary}\label{cor1}
For any hypergraph $\mathcal{H}$, there exist discrete Morse functions $\overline f$ on $\Delta\mathcal{H}$, $f$ on $\mathcal{H}$, and ${\underline   f}$ on $\delta\mathcal{H}$ such that $f=\overline f\mid_{\mathcal{H}}$  and ${\underline   f}= f\mid_{\delta\mathcal{H}}=\overline f\mid_{\delta\mathcal{H}}$.  \qed
\end{corollary}

%We may assume that $f$ and $\underline  f$ are injective as well.  

%\subsection{Critical Hyperedges}\label{subs-3.2}

%\section{Critical Hyperedges}

 %   In particular, if $\mathcal{H}$ is a simplicial complex, then 

%\smallskip

%\smallskip

%In the remaining part of this section, we give some examples of discrete Morse functions on hypergraphs. 
%\subsection{Some Examples}

In \cite[Lemma~2.5]{forman1}, it is proved that for a discrete Morse function $f$ on a simplicial complex $\mathcal{K}$ and a simplex $\alpha\in\mathcal{K}$, the conditions (A) and (B)  cannot both be true.  However,  if we substitute $\mathcal{K}$ with a hypergraph $\mathcal{H}$ and let $f$ be a discrete Morse function on $\mathcal{H}$, then the  next   example shows that  for certain hyperedge $\alpha\in\mathcal{H}$,  both (A) and (B) can  be true. 

\begin{example}\label{ex-2.a}
Let $\mathcal{H}=\{\{v_0\},\{v_0,v_1\},\{v_0,v_1,v_2\}\}$. Let 
\begin{eqnarray*}
f(\{v_0\})=2,  \text{\ \ \ }
f(\{v_0,v_1\})=1,   \text{\ \ \ }
f(\{v_0,v_1,v_2 \})=0.  
\end{eqnarray*}
Then $f$ is a discrete Morse function on $\mathcal{H}$, and $\{v_0,v_1\}$ is not critical. 
\begin{enumerate}[(i).]
\item
The hyperedge $\{v_0,v_1\}$ satisfies both of the conditions (A) and (B); % given in Subsection~\ref{subs-3.2}; 
\item
The discrete Morse function
$f$ cannot be extended to a discrete Morse function on $\Delta\mathcal{H}$. 
\end{enumerate}
\end{example}

%\begin{proof}%[Proof of Example~\ref{ex-2.a}]
In Example~\ref{ex-2.a},  it is direct to verify (i).  To prove (ii),  we suppose to the contrary,  $f$ is given by the restriction $f=\overline f\mid_{\mathcal{H}}$, where $\overline f:\Delta\mathcal{H}\longrightarrow \mathbb{R}$ is a discrete Morse function on $\Delta\mathcal{H}$. Then since the conditions (A) and (B) cannot both be true for $\overline f$ and any fixed  simplex  in $\Delta\mathcal{H}$, (A) and (B) also cannot both be true for $f$ and the hyperedges in $\mathcal{H}$.   This contradicts with (i).  Hence we obtain (ii). 
%\end{proof}

\begin{example}\label{le-2.a}
Suppose $\mathcal{H}$ satisfies the following condition 
\begin{quote}
(C). for any $n\geq 1$ and any hyperedges $\beta^{(n+1)}>\alpha^{(n)}>\gamma^{(n-1)}$ of $\mathcal{H}$, there exists $\hat   \alpha^{(n)}\in\mathcal{H}$, $\hat   \alpha\neq \alpha$,  such that $\beta>\hat  \alpha>\gamma$. 
\end{quote}
Then conditions (A) and (B) cannot both be true. %\qed
\end{example}

Example~\ref{le-2.a} can be obtained by an analogous  argument of  \cite[Lemma~2.5]{forman1}.   
In particular, if $\mathcal{H}$ is a simplicial complex, then (C) is satisfied and Example~\ref{le-2.a} gives  \cite[Lemma~2.5]{forman1}. % in the simplicial complex case. 

By Example~\ref{ex-2.a}~(ii), the extension property of discrete Morse functions (cf.  \cite[Lemma~4.2]{forman1} for the simplicial complex case) may not hold for general hypergraphs. The next example shows that the extension property even may not hold for the hypergraphs satisfying condition (C). 

\begin{example}\label{ex-a.a}
Let $\mathcal{H}=\{\{v_0\},\{v_1\},\{v_2\}, \{v_0,v_1,v_2\}\}$. Let $f(\{v_0\})=f(\{v_1\})=f(\{v_2\})=2$,  $f(\{v_0,v_1,v_2\})=0$.   Then $f$ is a discrete Morse function on $\mathcal{H}$.  All the hyperedges are critical.  Moreover,  $f$ cannot be extended to be a discrete Morse function on $\Delta\mathcal{H}$. 
\end{example}

\section{Discrete Gradient Vector Fields}\label{s4}

  In this section, we define the discrete gradient vector fields on $\mathcal{H}$. % and 
  %We prove some auxiliary lemmas.  %Then  we prove Lemma~\ref{co-4.a} and  Proposition~\ref{co-4.bb}. % and their relations with discrete Morse functions. 

\smallskip

%Firstly,  we give some definitions. 

\begin{definition}\label{def-a9}
A {\it discrete gradient vector field} $V$ on $\mathcal{H}$ is a collection of pairs $\{\alpha^{(n)}<\beta^{(n+1)}\}$   of hyperedges of $\mathcal{H}$,  $n\geq 0$,  such that  there is no nontrivial closed paths of the form
\begin{eqnarray*}
\alpha_0^{(n)}, \beta_0^{(n+1)}, \alpha_1^{(n)}, \beta_1^{(n+1)}, \alpha_2^{(n)}, \cdots, \alpha_r^{(n)}, \beta_r^{(n+1)}, \alpha_{r+1}^{(n)}=\alpha_0^{(n)},
\end{eqnarray*}
where for each $0\leq i\leq r$, $\{\alpha_i^{(n)}<\beta_i^{(n+1)}\}\in V$  and $\alpha_i\neq \alpha_{i+1} <\beta_i$. %, and $\alpha_0=\alpha_{r+1}$. 
In addition, if 
each hyperedge of $\mathcal{H}$ is in at most one pair in $V$, then we call $V$ a {\it proper discrete gradient vector field}. 
\end{definition}

Given a discrete gradient vector field $V$ on $\mathcal{H}$,  we can obtain an $R$-linear map \begin{eqnarray}\label{eq-3.a}
V: R( \mathcal{H})_n\longrightarrow R(\mathcal{H})_{n+1},
\end{eqnarray} 
where $V$ sends each $\alpha^{(n)}\in \mathcal{H}$ to  $-\langle\partial\beta,\alpha\rangle\beta$ if $\{\alpha^{(n)}<\beta^{(n+1)}\}\in V$, and sends $\alpha^{(n)}$  to $0$ otherwise.  Here $\langle\partial\beta,\alpha\rangle$ is the incidence number  of $\beta$ and $\alpha$ in the chain complex $C_*(\Delta\mathcal{H};R)$ (cf. \cite[p. 98]{forman1}).  Hence $\langle\partial\beta,\alpha\rangle$ takes the values $\pm 1$ where $1$ refers to the multiplicative unity of $R$.  We observe that $V$ is proper iff. as an $R$-linear map, 
\begin{eqnarray*}
V\circ V=0.
\end{eqnarray*}

 In particular, if $\mathcal{H}$ is a simplicial complex, then the proper discrete gradient vector fields defined in Definition~\ref{def3} and (\ref{eq-3.a}) are  exactly the discrete gradient vector fields defined in \cite[Definition~6.1]{forman1}. 

\smallskip

Let $f$ be a discrete Morse function on $\mathcal{H}$.  %We define the discrete gradient vector field of $f$. % is defined as follows. 

\begin{definition}\label{def3}
Let $\alpha^{(n)}\in\mathcal{H}$. If there exists $\beta^{(n+1)}>\alpha^{(n)}$, $\beta\in \mathcal{H}$,  such that $f(\beta)\leq f(\alpha)$, then we set 
\begin{eqnarray*}
V(\alpha)=-\langle\partial\beta,\alpha\rangle\beta.
\end{eqnarray*}
    If there does not exist such $\beta$, then we set 
  \begin{eqnarray*}
  V(\alpha)=0.
  \end{eqnarray*}
   We call $V$ the {\it discrete gradient vector field} of $f$ and write $V=\text{grad } f$.    
%Here $\langle \partial \beta,\alpha\rangle$ is  $1$ if $\alpha$ has the same orientation  
\end{definition}

We can regard $\text{grad } f$ defined in Definition~\ref{def3}  as the collection of  pairs  $\{\alpha^{(n)}<\beta^{(n+1)}\}$ of the hyperedges of $\mathcal{H}$ where  
\begin{eqnarray*}
(\text{grad } f ) \alpha=-\langle\partial\beta,\alpha\rangle\beta. 
\end{eqnarray*}
It can be verified that $\text{grad } f$ is a discrete gradient vector field on $\mathcal{H}$.  
The next example shows that the discrete gradient vector fields of  discrete Morse functions on hypergraphs may not be proper.

\begin{example}\label{ex-3.1}
Consider the hypergraph $\mathcal{H}$  and the discrete Morse function $f$ given in Example~\ref{ex-2.a}. 
Let $V$ be the discrete gradient vector field of $f$. Then 
\begin{eqnarray*}
V(\{v_0\})&=&\{v_0,v_1\},  \\
V(\{v_0,v_1\})&=& -\{v_0,v_1,v_2\}. 
\end{eqnarray*}
Hence $V\circ V (\{v_0\})= -\{v_0,v_1,v_2\}\neq 0$. 
\end{example}

 By a similar argument with   \cite[Theorem~6.3]{forman1}~(1),  we can prove that  if $\mathcal{H}$ satisfies condition (C) (cf. Example~\ref{le-2.a}),  then $\text{grad } f$ is a proper discrete gradient vector field.   
Moreover, if $\mathcal{H}$ is a simplicial complex,  then it is proved in \cite[Theorem~3.5]{forman2} that for any proper discrete gradient vector field $V$ defined in Definition~\ref{def-a9},  there exist discrete Morse functions $f: \mathcal{H}\longrightarrow \mathbb{R}$ such that $V=\text{grad } f$. 

\smallskip

\section{Auxiliary Results}

In this section, we prove some auxiliary results about discrete Morse functions and discrete gradient vector fields on hypergraphs.  The main result of this section is Proposition~\ref{co-4.bb}. 

\smallskip 

\begin{lemma}\label{lem2}
Let $\mathcal{H}$ and $\mathcal{H}'$ be two hyperedges such that $\mathcal{H}'\subseteq\mathcal{H}$. Let $f:\mathcal{H}\longrightarrow\mathbb{R}$ be a discrete Morse function on $\mathcal{H}$ and let $f'=f\mid_{\mathcal{H}'}$.  Then 
\begin{eqnarray*}
M(f,\mathcal{H})\cap\mathcal{H}'\subseteq M(f',\mathcal{H}').
\end{eqnarray*}  
\end{lemma}

\begin{proof}
The lemma follows from a  straight-forward verification by using Definition~\ref{def2}. 
%Let $\alpha^{(n)}\in M(f,\mathcal{H})\cap\mathcal{H}'$.  Then since $\alpha^{(n)}\in M(f,\mathcal{H})$, by Definition~\ref{def2}, 
%\begin{eqnarray*}
%\#\{\beta^{(n+1)}>\alpha^{(n)}\mid  f(\beta)\leq f(\alpha), \beta\in\mathcal{H}\}=0
%\end{eqnarray*}
%and 
%\begin{eqnarray*}
%\#\{\gamma^{(n-1)}<\alpha^{(n)}\mid  f(\gamma)\geq f(\alpha), \gamma\in\mathcal{H}\}=0. 
%\end{eqnarray*}
%Since $\mathcal{H}'\subseteq \mathcal{H}$, it follows that
%\begin{eqnarray*}
%\#\{\beta^{(n+1)}>\alpha^{(n)}\mid  f'(\beta)\leq f'(\alpha), \beta\in\mathcal{H}'\}=0
%\end{eqnarray*}
%and 
%\begin{eqnarray*}
%\#\{\gamma^{(n-1)}<\alpha^{(n)}\mid  f'(\gamma)\geq f'(\alpha), \gamma\in\mathcal{H}'\}=0. 
%\end{eqnarray*}
%Therefore, we obtain $\alpha\in M(f',\mathcal{H}')$. The lemma is proved. 
\end{proof}

The next lemma follows from Lemma~\ref{lem2} immediately. 

\begin{lemma}\label{pr-1.88}
Let $\overline f: \Delta\mathcal{H}\longrightarrow \mathbb{R}$ be a discrete Morse function on $\Delta\mathcal{H}$.  Let $f: \mathcal{H}\longrightarrow \mathbb{R}$  and $\underline  f: \delta\mathcal{H}\longrightarrow\mathbb{R}$ be the discrete Morse functions induced from $\overline f$.  Then 
%\begin{enumerate}[(i).]
%\item
%Let $\overline f: \Delta\mathcal{H}\longrightarrow \mathbb{R}$ be a discrete Morse function on $\Delta\mathcal{H}$ and let $f=\overline f\mid_{\mathcal{H}}$. Then  
\begin{eqnarray}
M(\overline f,\Delta\mathcal{H})\cap \mathcal{H}&\subseteq& M(f,\mathcal{H}), \label{eq-1.z}\\
M(\overline f,\Delta\mathcal{H})\cap\delta\mathcal{H}&\subseteq& M({\underline   f},\delta\mathcal{H}),\nonumber\\
M(f,\mathcal{H})\cap \delta\mathcal{H}&\subseteq& M({\underline   f},\delta\mathcal{H}). \nonumber 
\end{eqnarray}
\qed
%\end{enumerate}
\end{lemma}

%If condition (C) is not satisfied, then $\text{grad } f$ may not be proper.  We have the next example. 

%Now we prove some auxiliary lemmas. 

\begin{lemma}\label{le-9.1}
Let $\mathcal{H}$ and $\mathcal{H}'$ be two hypergraphs such that $\mathcal{H}'\subseteq\mathcal{H}$.  Let $V$ be a proper discrete gradient vector field on $\mathcal{H}'$. Then $V$ is also a proper discrete gradient vector field on $\mathcal{H}$. 
\end{lemma}

\begin{proof}
Let $V$ be a proper discrete gradient vector field on $\mathcal{H}'$.  We generalize the Hasse diagram construction (cf. \cite[Section~6]{forman2}) of simplicial complexes and apply a similar argument to hypergraphs.  We construct a digraph $D_\mathcal{H}$ (resp. $D_{\mathcal{H}'}$) whose vertices are in 1-1 correspondence with the hyperedges of $\mathcal{H}$ (resp. $\mathcal{H}'$).  For any hyperedges $\alpha$ and $\beta$  of $\mathcal{H}$ (resp. of $\mathcal{H}'$), we assign   a directed edge in $D_\mathcal{H}$  (resp. in $D_{\mathcal{H}'}$) from $\beta$ to $\alpha$, denoted as $\beta\to\alpha$, iff. $\alpha^{(n)}<\beta^{(n+1)}$ for some $n\geq 0$.  For any directed edge $\beta^{(n+1)}\to \alpha^{(n)}$ in $D_{\mathcal{H}'}$, if  $\{\alpha^{(n)},\beta^{(n+1)}\}\in V$, then we reverse the direction of $\beta\to\alpha$ in $D_{\mathcal{H}'}$ and obtain a  new directed edge $\alpha\to\beta$.  We obtain a new digraph $D_{\mathcal{H}',V}$.  Let the vertices of $D_{\mathcal{H},V}$ be the vertices of $D_{\mathcal{H}}$.  Let 
\begin{eqnarray}\label{eq-888}
E(D_{\mathcal{H},V})= E(D_{\mathcal{H}',V})\cup  ( E(D_{\mathcal{H}})\setminus E(D_{\mathcal{H}'})),
\end{eqnarray}
where $E(-)$ denotes the set of directed edges of a digraph.  We notice that for any directed edge $\beta\to\alpha$ in  $E(D_{\mathcal{H}})\setminus E(D_{\mathcal{H}'})$,  at least one of $\alpha$ and $\beta$ is not a vertex of $D_{\mathcal{H}'}$.  Hence the union in (\ref{eq-888}) is a disjoint union.  Since there is no nontrivial closed directed path in $E(D_{{\mathcal{H}'},V})$,   there is no nontrivial closed directed path in $E(D_{\mathcal{H},V})$ as well. By a similar argument of \cite[Theorem~6.2]{forman2}, it follows that  $V$ is a discrete gradient vector field on $\mathcal{H}$. Moreover,  each hyperedge of $\mathcal{H}$ appears in at most one pair  in $V$.   Hence $V$ is a proper discrete gradient vector field on $\mathcal{H}$. 
\end{proof}

The next lemma is a consequence of Lemma~\ref{le-9.1}.

\begin{lemma}\label{co-9.1}
Let $V$ be a proper discrete gradient vector field on $\mathcal{H}$. Then $V$ extends to   a proper discrete gradient vector field $\overline V$ on $\Delta\mathcal{H}$ such that $\overline V\mid _\mathcal{H}=V$ and  $\overline V\mid_{\Delta\mathcal{H}\setminus\mathcal{H}}=0$. \qed
\end{lemma}

The next lemma follows from Lemma~\ref{co-9.1}.

\begin{lemma}\label{pr-9.1}
Suppose $\mathcal{H}$ is a hypergraph satisfying condition (C) (cf.  Example~\ref{le-2.a}). 
Let $g$ be a discrete Morse function on $\mathcal{H}$.  Then there exists a discrete Morse function $\overline f$ on $\Delta\mathcal{H}$ such that 
\begin{enumerate}[(i).]
\item
$\text{grad } f=\text{grad } g$ where $f=\overline f\mid_\mathcal{H}$; % has the same gradient vector field with $g$;
\item
 $(\text{grad } \overline f)\mid_{\Delta\mathcal{H}\setminus\mathcal{H}}=0$. 
\end{enumerate}
\end{lemma}

\begin{proof}
Since $\mathcal{H}$ satisfies condition (C), by Example~\ref{le-2.a},  $\text{grad } g$ is a proper discrete gradient vector field on $\mathcal{H}$.  
By Lemma~\ref{co-9.1},  $\text{grad } g$ extends to a proper discrete gradient vector field $\overline {\text{grad }g }$ on $\Delta\mathcal{H}$ such that
\begin{eqnarray*}
\overline{\text{grad }g} \mid_\mathcal{H}=\text{grad } g
\end{eqnarray*}
and
\begin{eqnarray*}
\overline{\text{grad } g}\mid_{\Delta\mathcal{H}\setminus\mathcal{H}}=0. 
\end{eqnarray*}
By \cite[Theorem~3.5]{forman2},  there exists a discrete Morse function $\overline f$ on $\Delta\mathcal{H}$ such that $\overline {\text{grad } g}=\text{grad }\overline f$.  Hence 
\begin{eqnarray*}
\text{grad } f = \text{grad } (\overline f\mid_{\mathcal{H}})
= (\text{grad }\overline f)\mid_{\mathcal{H}}=\overline {\text{grad } g}\mid_\mathcal{H}
= {\text{grad } g}, 
\end{eqnarray*}
and
\begin{eqnarray*}
\text{grad }{\overline f}\mid_{\Delta\mathcal{H}\setminus\mathcal{H}}=\overline{\text{grad } g}\mid_{\Delta\mathcal{H}\setminus\mathcal{H}}=0. 
\end{eqnarray*}
The assertion is obtained. 
\end{proof}

By Lemma~\ref{pr-9.1}, in order to study the  discrete gradient vector fields of discrete Morse functions on hypergraphs satisfying condition (C),  it is sufficient to study the discrete gradient vector fields of discrete Morse functions on  associated simplicial complexes as well as the restrictions of the discrete gradient vector fields to the hypergraphs.

\smallskip

%The next proposition is an analogue of \cite[Theorem~6.3]{forman1}~(1). 

%\begin{proposition}\label{pr-3.x}
%Let $f$ be a discrete Morse function on $\mathcal{H}$. Suppose for every $p\geq 1$ and every $\alpha^{(n)}\in\mathcal{H}$, $\alpha$ satisfies condition (C) (cf. Example~\ref{le-2.a}). Then $\text{grad } f\circ \text{grad } f=0$.  \qed
%\end{proposition}

\begin{lemma}\label{le-4.x}
Let $\mathcal{H}$ and $\mathcal{H}'$ be two hypergraphs such that $\mathcal{H}'\subseteq\mathcal{H}$.  Let $f$ and $f'$  be discrete Morse functions on $\mathcal{H}$ and $\mathcal{H}'$ respectively such that $f'=f\mid_{\mathcal{H}'}$.  Let $V=\text{grad } f$ and $V'=\text{grad }f'$.  Then 
\begin{eqnarray}\label{eq-4.81}
V'=\pi(\mathcal{H},\mathcal{H}')\circ V
\end{eqnarray}
where $\pi(\mathcal{H},\mathcal{H}')$ is the canonical projection from $R(\mathcal{H})_*$ to $R(\mathcal{H}')_*$ sending an  $R$-linear combination of hyperedges
\begin{eqnarray*}
\sum_{i} x_i\sigma_i + \sum _{j}y_j \tau_j, 
\end{eqnarray*} 
where $\sigma_i\in\mathcal{H}'$, $\tau_j\in\mathcal{H}\setminus \mathcal{H}'$  and  $x_i, y_j\in R$, 
to the $R$-linear combination of hyperedges
\begin{eqnarray*}
\sum _{i} x_i\sigma_i.
\end{eqnarray*}
\end{lemma}

\begin{proof}
Let $\alpha^{(n)}, \beta^{(n+1)}\in \mathcal{H}$.  Then
\begin{eqnarray}
\label{eq-3.1.z}
%&&
\alpha,\beta\in\mathcal{H}' \text { and  }  V'(\alpha)=-\langle\partial\beta,\alpha\rangle\beta %\nonumber\\
%\Longleftrightarrow && \alpha,\beta\in\mathcal{H}', \alpha < \beta \text{   and  }  f'(\beta)\leq f'(\alpha)\nonumber\\
%\Longleftrightarrow && \alpha,\beta\in\mathcal{H}',  \alpha<\beta  \text{   and  }  f(\beta)\leq f(\alpha)\nonumber\\
&\Longleftrightarrow& %&& 
\alpha,\beta\in\mathcal{H}'  \text{ and  }V(\alpha)=-\langle\partial\beta,\alpha\rangle\beta,\\  
%\end{eqnarray}
%and 
%\begin{eqnarray}
\label{eq-3.1.x}
%&& 
\alpha\in\mathcal{H}' \text { and  }  V'(\alpha)=0 %\nonumber\\
&\Longleftrightarrow& %&& 
\alpha \in\mathcal{H}'  \text{ and  }V(\alpha)=0.   
\end{eqnarray}
By extending $V$ linearly over $R$, we obtain (\ref{eq-4.81}) from (\ref{eq-3.1.z}) and (\ref{eq-3.1.x}). 
\end{proof}

 By Corollary~\ref{cor1},  we let $\overline f$ and ${\underline   f}$ be  discrete Morse functions on $\Delta\mathcal{H}$ and $\delta\mathcal{H}$ respectively such that $\underline f=\overline f\mid_{\delta\mathcal{H}}$.  Let $f=\overline f\mid_{\mathcal{H}}$ be the discrete Morse function on $\mathcal{H}$.  We let
 \begin{enumerate}[(i).]
 \item
  $V=\text{grad } f$  on $\mathcal{H}$, 
  \item
  $\overline V=\text{grad } \overline f$ on $\Delta\mathcal{H}$, 
  \item
  $\underline V= \text{grad } \underline f$ on $\delta\mathcal{H}$. 
  \end{enumerate}
%Let $\overline f$ and ${\underline   f}$ be the discrete Morse functions given in Corollary~\ref{cor1}. Let $\overline V$ and ${\underline   V}$ be the discrete gradient vector fields of $\overline f$ and ${\underline   f}$ respectively.  Then both  $\overline V$ and ${\underline   V}$ satisfy \cite[Theorem~6.3]{forman1}.  Suppose $f=\overline f\mid_{\mathcal{H}}$ and $V$ is the discrete gradient vector field of $f$.  
The next lemma follows from \cite[Theorem~6.3]{forman1} and Lemma~\ref{le-4.x}. 
   %We prove  that $V$  satisfies similar properties as  \cite[Theorem~6.3]{forman1}. 

\begin{lemma}\label{pr-5}
%Let $\overline f$ be a discrete Morse function on $\Delta\mathcal{H}$ and $f=\overline f\mid _{\mathcal{H}}$.  Let $V$ be the discrete gradient vector field of $f$.  Then
\begin{enumerate}[(i).]
\item
$V\circ V=0$; 
\item
$M(f,\mathcal{H})=\{\sigma\in\mathcal{H}\mid \sigma\notin \text{Im}(V) \text{\  and \ } V(\sigma)=0\}$; 
\item
$\#\{\gamma^{(n-1)}\in\mathcal{H}\mid  V(\gamma^{(n-1)})=\alpha \}\leq 1$ for any $\alpha^{(n)}\in \mathcal{H}$; 
\item
The following diagram commutes
\begin{eqnarray*}
\xymatrix{
C_n( \delta\mathcal{H};R)\ar[d]^{{\underline   V}} \ar[rr]^{\text{inclusion}}&& R(\mathcal{H})_n\ar[d]^{V} \ar[rr]^{\text{inclusion}} && C_n( \Delta\mathcal{H};R)\ar[d]^{\overline V} \\
C_{n+1}(\delta\mathcal{H};R)&& R(\mathcal{H})_{n+1}\ar[ll]_{\pi(\mathcal{H},\delta\mathcal{H})} && C_{n+1}(\Delta\mathcal{H};R)\ar[ll]_{\pi(\Delta\mathcal{H},\mathcal{H})}. 
}
\end{eqnarray*}
\end{enumerate}
\end{lemma}

\begin{proof}
The proofs  of  (i), (ii), (iii) are similar with the proofs of (1), (2), (3) of \cite[Theorem~6.3]{forman1} respectively.  The proof of (iv) follows from Lemma~\ref{le-4.x}. 
\end{proof}

%Now  we prove Lemma~\ref{co-4.a} and Proposition~\ref{co-4.bb}. %, which is the main aim of this section. 
  
  The next lemma follows from Lemma~\ref{pr-1.88} and Lemma~\ref{pr-5}.

  \begin{lemma}\label{co-4.a}
  Let $\sigma\in \mathcal{H}$. Then 
 \begin{eqnarray}\label{eq-4.s} 
\sigma\in M(f, \mathcal{H})\setminus (M(\overline f,\Delta\mathcal{H})\cap \mathcal{H})
\end{eqnarray}
 iff.  one of the followings holds% all of the following three conditions hold
  \begin{enumerate}[(a).]
 % \item $\sigma\in \mathcal{H}$;
  \item
 $\overline V(\sigma)=\pm \eta$ for some $\eta \in \Delta\mathcal{H}\setminus \mathcal{H}$, and    for any $\tau\in \Delta\mathcal{H}$,   $\overline V(\tau)\neq\pm\sigma$; 
   \item
$\overline V(\sigma)=\pm\eta$ for some $\eta\in \Delta\mathcal{H}\setminus \mathcal{H}$,   and   there exists $\tau\in \Delta\mathcal{H}\setminus \mathcal{H}$ such that $\overline V(\tau)=\pm\sigma$;
   \item
$\overline V(\sigma)=0$,   and   there exists $\tau\in \Delta\mathcal{H}\setminus \mathcal{H}$ such that $\overline V(\tau)=\pm\sigma$. 
  \end{enumerate}
  \end{lemma}
  
  \begin{proof}
  Let $\sigma\in \mathcal{H}$.   By Lemma~\ref{pr-1.88},   the set complement  in  (\ref{eq-4.s}) is well-defined. 
By Lemma~\ref{pr-5}~(ii), 
\begin{eqnarray*}
%&&
\sigma\in M(\overline f,\Delta\mathcal{H})\cap \mathcal{H}%\\
\Longleftrightarrow %&&
  \overline V(\sigma)=0 \text{ and }  \overline V(\tau)\neq \pm\sigma \text{ for any } \tau\in \Delta\mathcal{H}.
\end{eqnarray*}
 %iff.  
 %\begin{eqnarray*}
 %\overline V(\sigma)=0
 %\end{eqnarray*}
 %and for any $\tau\in \Delta\mathcal{H}$,  
 %\begin{eqnarray*}
 %\overline V(\tau)\neq \pm\sigma.
 %\end{eqnarray*}
   Hence   
   \begin{eqnarray*}
   \sigma\notin M(\overline f,\Delta\mathcal{H})\cap \mathcal{H}
   \end{eqnarray*}
    iff. either  
 \begin{eqnarray}\label{eq-4.6.1}
 \overline V(\sigma)\neq 0
 \end{eqnarray}
  or 
  \begin{eqnarray}\label{eq-4.6.2}
 \text{ there exists }\tau\in \Delta\mathcal{H} \text{ such that }\overline V(\tau)= \pm\sigma.
 \end{eqnarray}  
 On the other hand,  by Lemma~\ref{pr-5}~(ii), 
 \begin{eqnarray*}
 %&&
 \sigma\in M(f,\mathcal{H})%\\
\Longleftrightarrow  %&&
 V (\sigma)=0  \text{ and  }  V(\tau)\neq\pm\sigma \text{ for any }\tau\in \mathcal{H}.  
 \end{eqnarray*}
  % iff.   
   %\begin{eqnarray*}
   %\pi(\Delta\mathcal{H},\mathcal{H}) \circ 
% V (\sigma)=0
 %\end{eqnarray*}
  %and  for any $\tau\in \mathcal{H}$, 
  %\begin{eqnarray*}
 % V(\tau)\neq\pm\sigma.
  %\end{eqnarray*}
    By Lemma~\ref{pr-5}~(iv),  %for any $\sigma\in\mathcal{H}$, 
 \begin{eqnarray}\label{eq-4.7.1}
V(\sigma)=0&\Longleftrightarrow& \pi(\Delta\mathcal{H},\mathcal{H}) \circ  \overline V(\sigma)=0\nonumber\\
 &\Longleftrightarrow& \overline V(\sigma)=0 \text{ or } \overline V(\sigma)=\pm\eta \text{ for some } \eta\in \Delta\mathcal{H}\setminus \mathcal{H}; 
 \end{eqnarray}
 and for any $\tau\in\mathcal{H}$, 
 \begin{eqnarray}\label{eq-4.7.2}
  V(\tau)\neq\pm\sigma &\Longleftrightarrow&
  \pi(\Delta\mathcal{H},\mathcal{H})\circ \overline V(\tau)\neq\pm\sigma \nonumber\\
  &\Longleftrightarrow& 
   \overline V(\tau)\neq\pm\sigma. %\nonumber\\
 %  &\Longleftrightarrow& 
  %  \overline V(\tau)=0 \text{ or }     \overline V(\tau)=\pm \eta \text{ for some }\eta\in \Delta\mathcal{H}\setminus\mathcal{H}.
   \end{eqnarray}
Therefore,  (\ref{eq-4.s})  holds  %$\sigma\in M(f, \mathcal{H})\setminus (M(\overline f,\Delta\mathcal{H})\cap \mathcal{H})$ 
iff.   both (\ref{eq-4.7.1}) and (\ref{eq-4.7.2}) hold,  and
   at least one of (\ref{eq-4.6.1}) and (\ref{eq-4.6.2}) holds;  iff.   one of the following cases is satisfied

   {\sc Case~1}. (\ref{eq-4.7.1}),  (\ref{eq-4.7.2}), and  (\ref{eq-4.6.1}) hold. Then  by (\ref{eq-4.7.1}),  $\overline V(\sigma)=\pm \eta$ for some $\eta \in \Delta\mathcal{H}\setminus \mathcal{H}$.  By (\ref{eq-4.7.2}),  we have the following subcases.

   {\sc Subcase~1.1}.  For any $\tau\in\Delta\mathcal{H}$, $\overline V(\tau)\neq\pm\sigma$.  Then we obtain (a).

   {\sc Subcase~1.2}.  There exists $\tau\in \Delta\mathcal{H}\setminus \mathcal{H}$ such that  $\overline V(\tau)=\pm\sigma$.  Then we obtain (b). 
   
   {\sc Case~2}.  both (\ref{eq-4.7.1}) and (\ref{eq-4.7.2}) hold, and (\ref{eq-4.6.1}) does not hold.  Then  (\ref{eq-4.6.2})  holds.  By (\ref{eq-4.7.2}),  we have $\tau\in \Delta\mathcal{H}\setminus\mathcal{H}$ in  (\ref{eq-4.6.2}).  We obtain  (c). 
 % We notice that (\ref{eq-4.6.1}),  (\ref{eq-4.7.1}) and (\ref{eq-4.7.2})  are equivalent to (a), and (\ref{eq-4.6.2}), (\ref{eq-4.7.1}) and (\ref{eq-4.7.2})  are equivalent to (b) and (c).  

Summarizing the above cases, it follows that  (\ref{eq-4.s})  holds iff.  one of (a), (b) and (c) is satisfied.  
     \end{proof}

  The next lemma is a consequence of Lemma~\ref{co-4.a}.

  \begin{lemma}\label{co-4.aa}
 Suppose $\overline V(\alpha)= V(\alpha)$ for $\alpha\in\mathcal{H}$ and $\overline V(\alpha)=0$ for $\alpha\in\Delta\mathcal{H}\setminus \mathcal{H}$.  Then
 \begin{eqnarray*}
 M(f,\mathcal{H})= M(\overline  f,\Delta\mathcal{H})\cap\mathcal{H}. 
 \end{eqnarray*} % Let $f$ be a discrete Morse function on $\mathcal{H}$ and let $V=\text{grad } f$. Suppose $V$ is proper. 
  \qed
  \end{lemma}
  
  We obtain the next proposition from Lemma~\ref{pr-9.1} and Lemma~\ref{co-4.aa}. 
  
  \begin{proposition}\label{co-4.bb}
Suppose $\mathcal{H}$  satisfies condition (C) (cf.  Example~\ref{le-2.a}). 
Let $g$ be a discrete Morse function on $\mathcal{H}$, $\overline f$ a discrete Morse function on $\Delta\mathcal{H}$ given in Lemma~\ref{pr-9.1}, and $f=\overline f\mid_{\mathcal{H}}$.  Then 
\begin{eqnarray*}
M(g,\mathcal{H})= M(f,\mathcal{H})= M(\overline  f,\Delta\mathcal{H})\cap\mathcal{H}. 
\end{eqnarray*}
  \end{proposition}
  
  \begin{proof}
  By Lemma~\ref{pr-9.1},    
   $\overline V(\alpha)= V(\alpha)$ for $\alpha\in\mathcal{H}$ and $\overline V(\alpha)=0$ for $\alpha\in\Delta\mathcal{H}\setminus \mathcal{H}$, where $V=\text{grad}~f=\text{grad}~g$.  
By Lemma~\ref{co-4.aa},    the assertion follows. 
  \end{proof}
  
  \begin{remark}\label{re-5.1}
  In Proposition~\ref{co-4.bb}, by substituting $\Delta\mathcal{H}$ with a general simplicial complex $\mathcal{K}$ such that $\mathcal{H}\subseteq\mathcal{K}$,  we will obtain 
  \begin{eqnarray}\label{eq-5.r.1}
M(g,\mathcal{H})= M(f,\mathcal{H})= M(  \overline f_\mathcal{K}, \mathcal{K})\cap\mathcal{H}. 
\end{eqnarray}
The proof of (\ref{eq-5.r.1}) is similar with Proposition~\ref{co-4.bb}. 
  \end{remark}

\section{Discrete Gradient Flows}\label{s5}

In this section,  we study the discrete gradient flows on the associated simplicial complex $\Delta\mathcal{H}$.  We use the discrete gradient flows to calculate the embedded homology $H_*(\mathcal{H})$ in Theorem~\ref{main-1}.  

\smallskip

By \cite[Definition~6.2]{forman1}, the discrete gradient flow  of $\Delta\mathcal{H}$ is defined by
\begin{eqnarray}\label{eq-202006}
&\overline\Phi= \text{Id} + \partial \overline V+ \overline V\partial,
\end{eqnarray}
which is an $R$-linear map
\begin{eqnarray*}
&\overline\Phi: C_n(\Delta\mathcal{H};R)\longrightarrow C_n(\Delta\mathcal{H};R), \text{ \ \ \ } n\geq 0. 
\end{eqnarray*} 
Let
\begin{eqnarray*}
C_*^{\overline\Phi}(\Delta\mathcal{H};R)=\Big\{\sum_{i}n_i \alpha_i\mid  \overline\Phi(\sum_i n_i\alpha_i)=\sum_i n_i\alpha_i, n_i\in R, \alpha_i\in \Delta\mathcal{H}\Big\}
\end{eqnarray*}
be the sub-chain complex of $C_*(\Delta\mathcal{H};R)$ consisting of all $\overline \Phi$-invariant chains (cf. \cite[p. 119]{forman1}).  
By \cite[Theorem~7.2]{forman1},  there exists a positive integer $N$ large enough such that 
\begin{eqnarray}\label{stab}
\overline\Phi^N=\overline\Phi^{N+1}=\overline\Phi^{N+2}=\cdots
\end{eqnarray}
We denote the stabilized map in (\ref{stab}) as $\overline\Phi^\infty$.  By the proof of \cite[Theorem~7.3]{forman1},  we have the following chain homotopy
\begin{eqnarray*}
\overline\Phi^\infty:  C_*(\Delta\mathcal{H};R)&\longrightarrow& C_*^{\overline\Phi}(\Delta\mathcal{H};R),\\
i: C_*^{\overline\Phi}(\Delta\mathcal{H};R)&\longrightarrow& C_*(\Delta\mathcal{H};R).  
\end{eqnarray*} 
Here $i$ is the canonical inclusion. It is proved in \cite[Theorem~7.3]{forman1} that 
\begin{eqnarray}
\label{eq-d.1}
\overline\Phi^\infty\circ i=\text{Id}
\end{eqnarray}
   and 
  {\color{black}  
   \begin{eqnarray}
   \label{eq-d.2}
   i\circ\overline\Phi^\infty\simeq \text{Id}.
   \end{eqnarray}  
   Here (\ref{eq-d.2}) is a chain homotopy 
   \begin{eqnarray*}
   \text{Id}-i\circ  \overline\Phi^\infty=\partial L+ L\partial
   \end{eqnarray*}
    for 
   \begin{eqnarray*}
   L=-\overline{V}\circ (\text{Id}+\overline\Phi+ \cdots + \overline\Phi^{N-1}). 
   \end{eqnarray*}
   A sub-chain complex $C'_*$  of $C_*(\Delta\mathcal{H};R)$ is called {\it $\overline{V}$-invariant}, if  for each $n\geq 0$, 
   \begin{eqnarray*}
   \overline{V}(C'_n)\subseteq C'_{n+1}. 
   \end{eqnarray*}
   Note that by (\ref{eq-202006}), for a $\overline{V}$-invariant sub-chain complex $C'_*$,   we have 
   \begin{eqnarray}\label{eq-202006.1}
   {\overline{\Phi}} (C'_*)\subseteq C'_*. 
   \end{eqnarray}

   \begin{lemma}\label{le-a}
Let $C'_*$ be a $\overline{V}$-invariant sub-chain complex of $C_*(\Delta\mathcal{H};R)$.  Then  we have two maps
\begin{eqnarray}
\label{eq-d.5}
\overline\Phi^\infty\mid_{C'_*}&: & C'_*\longrightarrow C'_*\cap  C^{\overline\Phi}_*(\Delta\mathcal{H};R),\\
\label{eq-d.6}
i\mid_{C'_*\cap C^{\overline\Phi}_*(\Delta\mathcal{H};R)}&: & C'_*\cap C^{\overline\Phi}_*(\Delta\mathcal{H};R)\longrightarrow C'_*,
\end{eqnarray}
which satisfy
\begin{eqnarray}
(\overline\Phi^\infty\mid_{C'_*})\circ (i\mid_{C'_*\cap C^{\overline\Phi}_*(\Delta\mathcal{H};R)})&=&\text{Id}, \label{eq-d.3}
\\
 (i\mid_{C'_*\cap C^{\overline\Phi}_*(\Delta\mathcal{H};R)})\circ(\overline\Phi^\infty\mid_{C'_*})&\simeq &\text{Id}  \label{eq-d.4}
  \end{eqnarray}
  and give a chain homotopy. 
\end{lemma}
\begin{proof}
Note that restricted to $C'_*$, both $\partial$ and $\overline{V}$ are well-defined.  
By (\ref{eq-d.1}) and (\ref{eq-d.2}), once the maps $\overline\Phi^\infty\mid_{C'_*}$ in (\ref{eq-d.5}) and $i\mid_{C'_*\cap C^{\overline\Phi}_*(\Delta\mathcal{H};R)}$ in (\ref{eq-d.6}) are proved to be well-defined,  then these two maps have to satisfy (\ref{eq-d.3}) and (\ref{eq-d.4}).  We notice that  the map (\ref{eq-d.6}) is well-defined.  Hence we only need to prove that (\ref{eq-d.5}) is well-defined.  
Let $\alpha\in C'_*$. %Since
%\begin{eqnarray*}
%i\circ \overline\Phi^\infty(\alpha)=\alpha,
%\end{eqnarray*}
By (\ref{eq-202006.1}), we have 
\begin{eqnarray}\label{eq-a.1}
 \overline\Phi^\infty(\alpha)\in C'. 
\end{eqnarray}
On the other hand,  
\begin{eqnarray}\label{eq-a.2}
 \overline\Phi^\infty(\alpha)\in C^{\overline\Phi}_*(\Delta\mathcal{H};R).  
\end{eqnarray}
It follows from (\ref{eq-a.1}) and (\ref{eq-a.2}) that 
\begin{eqnarray*}
 \overline\Phi^\infty(\alpha)\in C'_*\cap  C^{\overline\Phi}_*(\Delta\mathcal{H};R).
 \end{eqnarray*}
Hence the map 
$ \overline\Phi^\infty\mid _{C'_*}$ in (\ref{eq-d.5}) is well-defined.   
\end{proof}

}

\begin{lemma}\label{le-3.1}
For  $n\geq 0$, 
\begin{eqnarray}
\overline\Phi^\infty \text{Inf}_n(R(\mathcal{H})_*)&\subseteq&\text{Inf}_n(\overline\Phi^\infty R(\mathcal{H})_*), 
\label{eq-4.5}\\
\overline\Phi^\infty \text{Sup}_n(R(\mathcal{H})_*)&=&\text{Sup}_n(\overline\Phi^\infty R(\mathcal{H})_*).  
\label{eq-4.6}
\end{eqnarray}
Moreover,  if % either 
%  \begin{eqnarray}\label{eq-4.38}
% \text{Inf}_{n }(R(\mathcal{H})_*) \subseteq C_{n}^{\overline\Phi}(\Delta\mathcal{H};R)  
% \end{eqnarray}
% or
 \begin{eqnarray}\label{eq-4.18}
 \text{Sup}_{n-1}(R(\mathcal{H})_*) \subseteq C_{n-1}^{\overline\Phi}(\Delta\mathcal{H};R),     
 \end{eqnarray} 
then the equality holds in (\ref{eq-4.5}).  
\end{lemma}
\begin{proof}
Firstly,  we prove (\ref{eq-4.5}) and (\ref{eq-4.6}).   Let $\alpha\in \text{Inf}_n(R(\mathcal{H})_*)$. Then 
\begin{eqnarray}\label{eq-4.13}
\alpha \in R(\mathcal{H})_n
\end{eqnarray}
 and 
 \begin{eqnarray}\label{eq-4.14}
 \partial_n \alpha\in R(\mathcal{H})_{n-1}. 
 \end{eqnarray}
  It follows from (\ref{eq-4.13}) that  for any $N\geq 1$, 
\begin{eqnarray}\label{eq-4.10}
\overline\Phi^N\alpha\in \overline\Phi^N R(\mathcal{H})_n. 
\end{eqnarray}
On the other hand, since $\overline\Phi\partial_n=\partial_n\overline\Phi$ (cf. \cite[Theorem~6.4~(i)]{forman1}),  it follows from (\ref{eq-4.14})  that 
 \begin{eqnarray}\label{eq-4.11}
 \partial_n\overline\Phi^N\alpha=\overline\Phi^N\partial_n\alpha\in \overline\Phi^N R(\mathcal{H})_{n-1}.  
 \end{eqnarray}
Therefore,  it follows from (\ref{eq-4.10}) and (\ref{eq-4.11}) that
\begin{eqnarray}\label{eq-4.12}
\overline\Phi^N\alpha\in (\overline\Phi^N R(\mathcal{H})_n)\cap \partial_n^{-1}(\overline\Phi^N R(\mathcal{H})_{n-1})=\text{Inf}_n (\overline\Phi^{N} R(\mathcal{H})_*). 
\end{eqnarray}
From  (\ref{stab}) and (\ref{eq-4.12}),  we obtain
\begin{eqnarray*}
\overline\Phi^\infty \alpha \in \text{Inf}_n (\overline\Phi^{\infty} R(\mathcal{H})_*). 
\end{eqnarray*}
Consequently,  we obtain (\ref{eq-4.5}).  
%\begin{eqnarray*}
%\overline\Phi^\infty \text{Inf}_n(R(\mathcal{H})_*)\subseteq \text{Inf}_n(\overline\Phi^\infty R(\mathcal{H})_*). 
%\end{eqnarray*}
 The assertion (\ref{eq-4.6})  follows from a direct calculation
\begin{eqnarray*}
\overline\Phi^\infty \text{Sup}_n(R(\mathcal{H})_*)&=&\overline\Phi^\infty R(\mathcal{H})_n+ \overline\Phi^\infty\partial_{n+1} R(\mathcal{H})_{n+1}\\
&=&\overline\Phi^\infty R(\mathcal{H})_n+ \partial_{n+1}\overline\Phi^\infty R(\mathcal{H})_{n+1}\\
&=&\text{Sup}_n(\overline\Phi^\infty R(\mathcal{H})_*).  
\end{eqnarray*}

Secondly,  suppose %either (\ref{eq-4.38}) or
 (\ref{eq-4.18}) holds. %In order to 
 We will prove the equality in (\ref{eq-4.5}). %,  we consider the following two cases.  
%{\sc Case~1}. The condition (\ref{eq-4.38}) holds. Then for each  $\alpha\in  \text{Inf}_{n }(R(\mathcal{H})_*)$,
% $\overline\Phi(\alpha)=\alpha$. Hence the equality holds in (\ref{eq-4.5}).   
%{\sc Case~2}. The condition (\ref{eq-4.18}) holds.  
%Then 
We notice that  (\ref{eq-4.18}) implies that the map 
 \begin{eqnarray*}
\overline\Phi^\infty\mid_{\text{Sup}_{n-1}(R(\mathcal{H})_*)}:  \text{Sup}_{n-1}(R(\mathcal{H})_*)\longrightarrow \text{Sup}_{n-1}(R(\mathcal{H})_*) \cap C_{n-1}^{\overline\Phi}(\Delta\mathcal{H};R)
 \end{eqnarray*}
 is an isomorphism.  In other words, 
\begin{eqnarray}\label{eq-4.28}
\text{Sup}_{n-1}(R(\mathcal{H})_*)\cap(\text{Ker} \overline\Phi^\infty)=\{0\}. 
\end{eqnarray}
Let 
$\alpha' \in \text{Inf}_n(\overline\Phi^\infty R(\mathcal{H})_*)$.    
We have 
\begin{eqnarray}\label{eq-4.1}
\alpha'=\overline\Phi^\infty \beta
\end{eqnarray}
for some $\beta\in R(\mathcal{H})_n$, and 
\begin{eqnarray}\label{eq-4.2}
\partial\alpha'=\overline\Phi^\infty \gamma
\end{eqnarray}
for some $\gamma\in R(\mathcal{H})_{n-1}$. 
Since    $\overline\Phi\partial_n=\partial_n\overline\Phi$, it follows from (\ref{eq-4.1}) that
\begin{eqnarray}\label{eq-4.3}
\partial\alpha'=\partial\overline\Phi^\infty\beta=\overline\Phi^\infty\partial\beta. 
\end{eqnarray}
Combining (\ref{eq-4.2}) and (\ref{eq-4.3}), it follows that
\begin{eqnarray}\label{eq-4.9}
\overline\Phi^\infty (\partial\beta-\gamma)=0. 
\end{eqnarray}
On the other hand,  we notice that  
\begin{eqnarray}\label{eq-4.19}
\partial\beta-\gamma\in \text{Sup}_{n-1}(R(\mathcal{H})_*).  
\end{eqnarray}
Hence by (\ref{eq-4.28}), (\ref{eq-4.9}) and  (\ref{eq-4.19}),  
\begin{eqnarray}\label{eq-4.20}
\partial\beta-\gamma=0. 
\end{eqnarray}
It follows from (\ref{eq-4.20}) that 
\begin{eqnarray*}
\partial\beta=\gamma\in R(\mathcal{H})_{n-1}, 
\end{eqnarray*}
which implies 
\begin{eqnarray}\label{eq-4.21}
\beta\in \text{Inf}_n(R(\mathcal{H})_*). 
\end{eqnarray}
Therefore, by (\ref{eq-4.1}) and (\ref{eq-4.21}), we have
\begin{eqnarray*}
\alpha'\in \overline\Phi^\infty  \text{Inf}_n(R(\mathcal{H})_*).  
\end{eqnarray*}
The equality holds in (\ref{eq-4.5}).  
\end{proof}

 \smallskip

Given a hypergraph $\mathcal{H}$, we consider the following sub-chain complexes  of $C_*(\Delta\mathcal{H};R)$:  
\begin{eqnarray*}
\text{Inf}_n^{\overline\Phi}(R(\mathcal{H})_*)&=&\text{Inf}_n (R(\mathcal{H})_*)  \cap  C^{\overline\Phi}_*(\Delta\mathcal{H};R),\\
\text{Sup}_n^{\overline\Phi}(R(\mathcal{H})_*)&=&\text{Sup}_n (R(\mathcal{H})_*)  \cap  C^{\overline\Phi}_*(\Delta\mathcal{H};R). 
\end{eqnarray*} 
By the proof of Lemma~\ref{le-a}, 
\begin{eqnarray}
 \text{Inf}_n^{\overline\Phi}(R(\mathcal{H})_*)&=&\overline\Phi^\infty \text{Inf}_n (R(\mathcal{H})_*),\label{eq-c.1}\\
  \text{Sup}_n^{\overline\Phi}(R(\mathcal{H})_*)&=&\overline\Phi^\infty \text{Sup}_n (R(\mathcal{H})_*). \label{eq-c.2}
 \end{eqnarray}
With the helps of Lemma~\ref{le-a} and Lemma~\ref{le-3.1}, we have the next theorem. 

{\color{black}
\begin{theorem}\label{main-1}
Let $\mathcal{H}$ be a hypergraph and $n\geq 0$. 
\begin{enumerate}[(i). ]
\item 
If $\text{Inf}_*(\mathcal{H})$ is $\overline {V}$-invariant, then 
\begin{eqnarray*}
H_n(\mathcal{H};R)\cong H_n(\text{Inf}_*^{\overline\Phi}(R(\mathcal{H})_*)).
\end{eqnarray*}
Moreover,  if  %either (\ref{eq-4.38}) or 
(\ref{eq-4.18}) is satisfied, then
\begin{eqnarray*}
H_n(\mathcal{H};R)\cong   H_n(\text{Inf}_*(\overline\Phi^\infty R(\mathcal{H})_*)). 
\end{eqnarray*}
\item
If $\text{Sup}_*(\mathcal{H})$ is $\overline {V}$-invariant, then 
\begin{eqnarray*}
H_n(\mathcal{H};R)&\cong& H_n(\text{Sup}_*^{\overline\Phi}(R(\mathcal{H})_*))\\
&\cong& H_n(\text{Sup}_*(\overline\Phi^\infty R(\mathcal{H})_*)). 
\end{eqnarray*}

\end{enumerate}
\end{theorem}

\begin{proof}
(i).  Suppose $\text{Inf}_*(\mathcal{H})$ is $\overline {V}$-invariant.  Then 
By Lemma~\ref{le-a}, there is a chain homotopy between $\text{Inf}_*(R(\mathcal{H})_*)$ and  $\text{Inf}^{\overline\Phi}_*(R(\mathcal{H})_*)$. Hence
\begin{eqnarray}\label{eq-b.1}
H_*(\text{Inf}_*(R(\mathcal{H})_*))\cong H_*(\text{Inf}^{\overline\Phi}_*(R(\mathcal{H})_*)). 
\end{eqnarray}
In addition, suppose %either (\ref{eq-4.38}) or 
 (\ref{eq-4.18}) is satisfied.  Then by Lemma~\ref{le-3.1},  we have
\begin{eqnarray}\label{eq-c.7}
\overline\Phi^\infty \text{Inf}_n(R(\mathcal{H})_*)=\text{Inf}_n(\overline\Phi^\infty R(\mathcal{H})_*).  
\end{eqnarray}
By (\ref{eq-c.1})%,  (\ref{eq-b.3}) 
and (\ref{eq-c.7}), the  assertion (i) follows. 

(ii).  Suppose $\text{Sup}_*(\mathcal{H})$ is $\overline {V}$-invariant.  
Also by Lemma~\ref{le-a}, there is a chain homotopy between $\text{Sup}_*(R(\mathcal{H})_*)$ and  $\text{Sup}^{\overline\Phi}_*(R(\mathcal{H})_*)$. Hence
\begin{eqnarray}\label{eq-b.2}
H_*(\text{Sup}_*(R(\mathcal{H})_*))\cong H_*(\text{Sup}^{\overline\Phi}_*(R(\mathcal{H})_*)). 
\end{eqnarray}
By %(\ref{eq-b.1}) and 
(\ref{eq-b.2}), we have
\begin{eqnarray}\label{eq-b.3}
H_n(\mathcal{H};R) %\cong  H_n(\text{Inf}_*^{\overline\Phi}(R(\mathcal{H})_*)) 
 \cong  H_n(\text{Sup}_*^{\overline\Phi}(R(\mathcal{H})_*)). 
\end{eqnarray}
On the other hand, by  (\ref{eq-4.6}) and (\ref{eq-c.2}),  we have
\begin{eqnarray}\label{eq-c.6}
\text{Sup}_*^{\overline\Phi}(R(\mathcal{H})_*)=\text{Sup}_*(\overline\Phi^\infty R(\mathcal{H})_*). 
\end{eqnarray}
Hence by (\ref{eq-b.3}) and (\ref{eq-c.6}),  we have
\begin{eqnarray*}
H_n(\mathcal{H};R) \cong  H_n(\text{Sup}_*(\overline\Phi^\infty R(\mathcal{H})_*)). 
\end{eqnarray*}
The  assertion (ii)  follows.  
\end{proof}

\begin{remark}
By Lemma~\ref{co-9.1}, represented by arrows between simplices (hyperedges), $\overline V$ is $V$ on $\mathcal{H}$ and is zero on $\Delta\mathcal{H}\setminus \mathcal{H}$.  Hence $\overline {V}$  is completely determined by the discrete Morse function $f$ on $\mathcal{H}$.  Thus in Theorem~\ref{main-1},  we may write  the term  "$\overline{V}$-invariant" as "$\text{grad}~f$-invariant". 
\end{remark}
}
%\begin{example}

%\end{example}

\section{Critical Simplices and Critical Hyperedges}\label{s6}

In this section,  we use critical simplices of $\Delta\mathcal{H}$ and critical hyperedges of $\mathcal{H}$ to calculate the embedded homology.  we prove Theorem~\ref{th-0.05} and Theorem~\ref{co-0.08} for the special case that the simplicial complex $\mathcal{K}$ is $\Delta\mathcal{H}$,  in Theorem~\ref{th-6.5} and Corollary~\ref{co-6.888} respectively. %, for simplicity reasons. % For a general  simplicial complex $\mathcal{K}$ containing the hyperedges of $\mathcal{H}$ as simplices, the proof could be applied with minor modifications.

\smallskip

By \cite[Theorem~8.2]{forman1},  there is an isomorphism of graded $R$-modules
\begin{eqnarray}\label{eq-5.1}
\overline\Phi^\infty\mid _{R(M(\overline f,\Delta\mathcal{H}))_*}: R(M(\overline f,\Delta\mathcal{H}))_*\longrightarrow C_*^{\overline\Phi}(\Delta\mathcal{H};R).
\end{eqnarray}
The  next lemma follows.

\begin{lemma}\label{le-5.1}
Let $D_*$ be a graded sub-$R$-module of $ C_*(\Delta\mathcal{H};R)$.  Then we have an isomorphism of graded $R$-modules
\begin{eqnarray*}
\overline\Phi^\infty\mid_{R(M(\overline f,\Delta\mathcal{H}))_n\cap D_n}:   R(M(\overline f,\Delta\mathcal{H}))_n\cap D_n \longrightarrow  \overline\Phi^\infty D_n,  \text{\ \ \ } n\geq 0. 
\end{eqnarray*}
\end{lemma}
\begin{proof}
By a direct calculation, it follows from (\ref{eq-5.1}) that
\begin{eqnarray*}
\overline\Phi^\infty (R(M(\overline f,\Delta\mathcal{H}))_n\cap D_n )&=& \overline\Phi^\infty (R(M(\overline f,\Delta\mathcal{H}))_n)\cap \overline\Phi^\infty D_n \\
&=& C_*^{\overline\Phi}(\Delta\mathcal{H};R) \cap  \overline\Phi^\infty D_n \\
&=&  \overline\Phi^\infty D_n.  
\end{eqnarray*}
The lemma is proved.  
\end{proof}

The next lemma follows from Lemma~\ref{le-5.1}.  

\begin{lemma}\label{pr-5.1}
We have the following isomorphisms of graded $R$-modules:
\begin{enumerate}[(i).]
\item
%\begin{eqnarray*}
$  R(M(\overline f,\Delta\mathcal{H}))_n\cap \text{Inf}_n(R(\mathcal{H})_*) \overset{\overline\Phi^\infty}{\longrightarrow}  \overline\Phi^\infty \text{Inf}_n(R(\mathcal{H})_*)$,  $n\geq 0$;  
%\end{eqnarray*}
\item
%\begin{eqnarray*}
$  R(M(\overline f,\Delta\mathcal{H}))_n\cap \text{Sup}_n(R(\mathcal{H})_*) \overset{\overline\Phi^\infty}{\longrightarrow}  \overline\Phi^\infty \text{Sup}_n(R(\mathcal{H})_*)$,  $n\geq 0$. %;  
%\end{eqnarray*}
%\item
%$  R(M(\overline f,\Delta\mathcal{H})\cap\mathcal{H})_n\overset{\overline\Phi^\infty}{\longrightarrow}  \overline\Phi^\infty R(\mathcal{H})_n$,  $n\geq 0$;  

%\item
%$R(M(\overline f, \Delta\mathcal{H})\cap\delta\mathcal{H})_n \overset{\overline\Phi^\infty}{\longrightarrow}   \overline\Phi^\infty C_n(\delta\mathcal{H};R)$, $n\geq 0$. 
\end{enumerate}
\end{lemma}

\begin{proof}
By substituting $D_n$ in Lemma~\ref{le-5.1} with $\text{Inf}_n(R(\mathcal{H})_*)$ and $\text{Sup}_n(R(\mathcal{H})_*)$ respectively,   we obtain (i) and (ii).  %We notice that
%\begin{eqnarray}
%R(M(\overline f,\Delta\mathcal{H})\cap\mathcal{H})_n&=& R(M(\overline f,\Delta\mathcal{H})_n\cap R(\mathcal{H})_n,\label{eq-5.8}\\
%R(M(\overline f, \Delta\mathcal{H})\cap\delta\mathcal{H})_n&=&R(M(\overline f, \Delta\mathcal{H})_n\cap C_n(\delta\mathcal{H};R).
%\label{eq-5.7}
%\end{eqnarray}
%By substituting  $D_n$ in Lemma~\ref{le-5.1} with $R(\mathcal{H})_n$,  we obtain (iii) with the help of (\ref{eq-5.8}). And by substituting  $D_n$ in Lemma~\ref{le-5.1} with $C_n(\delta\mathcal{H};R)$,  we obtain (iv) with the help of (\ref{eq-5.7}). 
\end{proof}

For each $n\geq 0$,  we consider the natural projection
\begin{eqnarray*}
\pi_{M}: C_n(\Delta\mathcal{H};R)\longrightarrow  R(M(\overline f,\Delta\mathcal{H}))_n
\end{eqnarray*}
sending each critical simplex to itself and sending each non-critical simplex to zero.  Restricting $\pi_M$ to  $C_*^{\overline\Phi}(\Delta\mathcal{H};R)$,  we have an isomorphism of graded $R$-modules
\begin{eqnarray}\label{eq-abc}
\pi_{M}\mid_{C_*^{\overline\Phi}(\Delta\mathcal{H};R)}: C_*^{\overline\Phi}(\Delta\mathcal{H};R)\longrightarrow  R(M(\overline f,\Delta\mathcal{H}))_*. 
\end{eqnarray}
By \cite[Theorem~2.2]{forman3}, 
\begin{eqnarray*}
\pi_{M}\mid_{C_*^{\overline\Phi}(\Delta\mathcal{H};R)}=(\overline\Phi^{\infty}\mid_{R(M(\overline f,\Delta\mathcal{H}))_*})^{-1}. 
\end{eqnarray*}
The next theorem follows from Theorem~\ref{main-1} and Lemma~\ref{pr-5.1}~(i), (ii).

\begin{theorem}\label{th-6.5}
Let $\mathcal{H}$ be a hypergraph and $n\geq 0$.  {\color{black} Suppose both $\text{Inf}_*(\mathcal{H})$ and $\text{Sup}_*(\mathcal{H})$ are $\text{grad}~f$-invariant. } Then  the embedded homology of $\mathcal{H}$ satisfies 
the following isomorphisms of homology groups
\begin{eqnarray*}
H_n(\mathcal{H};R)&\cong& H_n(\{R(M(\overline f,\Delta\mathcal{H}))_k\cap \text{Inf}_k(R(\mathcal{H})_*),\tilde\partial_k\}_{k\geq 0})\\
&\cong&  H_n(\{R(M(\overline f,\Delta\mathcal{H}))_k\cap \text{Sup}_k(R(\mathcal{H})_*),\tilde\partial_k\}_{k\geq 0}). 
\end{eqnarray*}
Here 
\begin{eqnarray}\label{eq-6.v}
\tilde \partial_k=(\pi_{M}\mid_{C_*^{\overline\Phi}(\Delta\mathcal{H};R)})\circ \partial_k\circ  (\overline\Phi^\infty\mid _{R(M(\overline f,\Delta\mathcal{H}))_*})
\end{eqnarray}
 is the boundary map from $R(M(\overline f,\Delta\mathcal{H}))_k$ to $R(M(\overline f,\Delta\mathcal{H}))_{k-1}$. The explicit formula of $\tilde \partial_k$ is given in \cite[Theorem~8.10]{forman1}.   
\qed
\end{theorem}

%Theorem~\ref{th-6.5} is equivalent to Theorem~\ref{th-0.05}.  Taking $\mathcal{K}$ to be $\Delta\mathcal{H}$, we see that Theorem~\ref{th-0.05} implies Theorem~\ref{th-6.5}.  Conversely, Theorem~\ref{th-6.5} also implies Theorem~\ref{th-0.05}.  In fact,  for any simplicial complex $\mathcal{K}$ such that $\mathcal{H}\subseteq\mathcal{K}$,  it is proved in \cite[Proposition~3.2]{h1} that $\Delta\mathcal{H}\subseteq \mathcal{K}$.  Hence any  discrete Morse function $\overline f_{\mathcal{K}}$ on $\mathcal{K}$ gives a discrete Morse function $\overline f$ on $\Delta\mathcal{H}$. By (\ref{eq-2.99}),  for any $k\geq 0$,  

%Moreover,
%\begin{eqnarray*}
%\tilde\partial_k=\tilde \partial_k^{\mathcal{K}}\mid_{R(\Delta\mathcal{H})_k}. 
%\end{eqnarray*}
%Therefore, Theorem~\ref{th-6.5} implies Theorem~\ref{th-0.05}. 

\smallskip

The next corollary follows from Lemma~\ref{pr-1.88}  and Theorem~\ref{th-6.5}. 

\begin{corollary}\label{co-6.x}
Let $\mathcal{H}$ be a hypergraph  {\color{black} such both $\text{Inf}_*(\mathcal{H})$ and $\text{Sup}_*(\mathcal{H})$ are $\text{grad}~f$-invariant, } and $n\geq 0$.   Suppose there are  discrete Morse functions $\overline f$ on $\Delta\mathcal{H}$ and $f=\overline f\mid_{\mathcal{H}}$ on $\mathcal{H}$ such that the equality of (\ref{eq-1.z}) holds. Then
\begin{eqnarray}
H_n(\mathcal{H};R)\cong H_n(\{R(M(f,\mathcal{H}))_k\cap \partial_k^{-1}(R(\mathcal{H})_{k-1}),\tilde\partial_k\}_{k\geq 0})\label{eq-6.a}. %\\
%&\cong&  H_n(\{R(M( f,\mathcal{H}))_k,\tilde\partial_k\}_{k\geq 0}). \label{eq-6.b}
\end{eqnarray}
%Moreover, if $(\text{grad } \overline f )(\alpha)=(\text{grad } f) (\alpha)$ for $\alpha\in\mathcal{H}$ and  $(\text{grad } \overline f )(\alpha)=0$ for $\alpha\in\Delta\mathcal{H}\setminus \mathcal{H}$, then both (\ref{eq-6.a}) and (\ref{eq-6.b}) hold. 
%.  Then  
%\begin{eqnarray*}
%H_n(\mathcal{H};R)&\cong& H_n(\{R(M(f,\mathcal{H}))_k\cap \text{Inf}_k(R(\mathcal{H})_*),\tilde\partial_k\}_{k\geq 0})\\
%&\cong&  H_n(\{R(M( f,\mathcal{H}))_k,\tilde\partial_k\}_{k\geq 0}). 
%\end{eqnarray*}
\end{corollary}
\begin{proof}
Suppose the equality of (\ref{eq-1.z}) holds. Then for each $k\geq 0$,  
\begin{eqnarray}\label{eq-6.f}
R(M(\overline f,\Delta\mathcal{H}))_k\cap R(\mathcal{H})_k=R(M(f,\mathcal{H}))_k.
\end{eqnarray} 
Since  $ \text{Inf}_k(R(\mathcal{H})_*)\subseteq R(\mathcal{H})_k$ and $R(M(f,\mathcal{H}))_k\subseteq R(\mathcal{H})_k$, it follows from (\ref{eq-6.f}) that
\begin{eqnarray*}
&&R(M(f,\Delta\mathcal{H}))_k\cap \text{Inf}_k(R(\mathcal{H})_*)\\
%&=&R(M(\overline f,\Delta\mathcal{H}))_k\cap R(\mathcal{H})_k\cap \text{Inf}_k(R(\mathcal{H})_*)\\
&=&R(M(f,\mathcal{H}))_k\cap  \text{Inf}_k(R(\mathcal{H})_*)\\
&=&R(M(f,\mathcal{H}))_k\cap \partial_k^{-1}(R(\mathcal{H})_{k-1}).  
\end{eqnarray*}
Hence by Theorem~\ref{th-6.5},  we obtain (\ref{eq-6.a}). 
%the corollary follows. 
%Moreover, suppose  $(\text{grad } \overline f )(\alpha)=(\text{grad } f) (\alpha)$ for $\alpha\in\mathcal{H}$ and  $(\text{grad } \overline f )(\alpha)=0$ for $\alpha\in\Delta\mathcal{H}\setminus \mathcal{H}$.  Then by  Lemma~\ref{co-4.aa},  for each $k\geq 0$,  $R(M(\overline f,\Delta\mathcal{H}))_k=R(M(f,\mathcal{H}))_k$. 
%Since 
%\begin{eqnarray*}
%R(M(f,\mathcal{H}))_k\cap \text{Sup}_k(R(\mathcal{H})_*)=R(M(f,\mathcal{H}))_k,
%\end{eqnarray*}
%the corollary follows. 
\end{proof}

The next corollary follows from Lemma~\ref{co-4.aa} and Corollary~\ref{co-6.x}. 

\begin{corollary}
Let $\mathcal{H}$ be a hypergraph {\color{black} such both $\text{Inf}_*(\mathcal{H})$ and $\text{Sup}_*(\mathcal{H})$ are $\text{grad}~f$-invariant, }  and $n\geq 0$.   If $(\text{grad } \overline f )(\alpha)=(\text{grad } f) (\alpha)$ for $\alpha\in\mathcal{H}$ and  $(\text{grad } \overline f )(\alpha)=0$ for $\alpha\in\Delta\mathcal{H}\setminus \mathcal{H}$, then (\ref{eq-6.a})  holds. \qed
\end{corollary}

The next corollary follows from Proposition~\ref{co-4.bb} and Corollary~\ref{co-6.x}.

\begin{corollary}\label{co-6.888}
Let $\mathcal{H}$ be  a hypergraph satisfying condition (C) (cf.  Example~\ref{le-2.a}).  Let $g$ be a discrete Morse function on  $\mathcal{H}$ such that {\color{black}  both $\text{Inf}_*(\mathcal{H})$ and $\text{Sup}_*(\mathcal{H})$ are $\text{grad}~g$-invariant.}   Then 
\begin{eqnarray}\label{eq-uuu}
H_n(\mathcal{H};R)\cong H_n(\{R(M(g,\mathcal{H}))_k\cap \partial_k^{-1}(R(\mathcal{H})_{k-1}),\tilde\partial_k\}_{k\geq 0}). 
\end{eqnarray}
\qed
\end{corollary}

\smallskip

Theorem~\ref{th-6.5} can be used to calculate the embedded homology of hypergraphs. The following is a concrete example.

\begin{example}\label{ex-6.8}
Let $\mathcal{H}=\{\{v_0\},\{v_1\},\{v_2\},\{v_3\},\{v_0,v_1\},\{v_0,v_3\}, \{v_1,v_3\}, \{v_0,v_1,v_2\}\}$.  Then $\Delta\mathcal{H}$ is the simplicial complex obtained by adding the $1$-simplices
\begin{eqnarray*}
\{v_0,v_2\},   \{v_1,v_2\}
\end{eqnarray*}
to $\mathcal{H}$. 
Let $\overline f$  be a discrete Morse function on $\Delta\mathcal{H}$  given by
\begin{eqnarray*}
&\overline f(\{v_0\})=1,\\
& \overline f(\{v_1\})=\overline f(\{v_2\})=\overline f(\{v_3\})=0,\\
& \overline f (\{v_0,v_1\})=\overline f(\{v_1,v_2\})=\overline  f(\{v_1,v_3\})=1,\\
& \overline f(\{v_0,v_2\})=\overline f(\{v_0,v_3\})=2,\\
&\overline f (\{v_0,v_1,v_2\})=2. 
\end{eqnarray*}
Then  
\begin{eqnarray*}
M(\overline f,  \Delta\mathcal{H})=\{\{v_1\},\{v_2\},\{v_3\}, \{v_0,v_3\}, \{v_1,v_2\}, \{v_1,v_3\}\}, 
\end{eqnarray*}
and 
\begin{eqnarray*}
R(M(\overline f,\Delta\mathcal{H}))_0\cap \text{Inf}_0(R(\mathcal{H})_*)&=&R(\{v_1\},\{v_2\},\{v_3\}),\\
R(M(\overline f,\Delta\mathcal{H}))_1\cap \text{Inf}_1(R(\mathcal{H})_*)&=&R(\{v_0,v_3\},\{v_1,v_3\}),\\
R(M(\overline f,\Delta\mathcal{H}))_2\cap \text{Inf}_2(R(\mathcal{H})_*)&=&0. 
\end{eqnarray*}
Let $\overline V=\text{grad } \overline f$ be the discrete gradient vector field on $\Delta\mathcal{H}$. Then
\begin{eqnarray*}
&\overline V(\{v_0\})=\{v_0,v_1\},\\
&\overline V(\{v_0,v_2\})=\{v_0,v_1,v_2\},\\
&\overline V(\sigma)=0 \text{ for any }\sigma\in \Delta\mathcal{H}\setminus\{\{v_0\}, \{v_0,v_2\}\}.
\end{eqnarray*}
Let $\overline \Phi=\text{Id}+ \partial \overline V+\overline V\partial$ be the discrete gradient flow on $\Delta\mathcal{H}$. Then 
\begin{eqnarray*}
&\overline \Phi(\{v_0\})=\{v_1\},&\overline \Phi(\{v_1\})=\{v_1\},\\
&\overline \Phi(\{v_2\})=\{v_2\},  
&\overline \Phi(\{v_3\})=\{v_3\},\\
&\overline \Phi(\{v_1,v_2\})=\{v_1,v_2\}, 
&\overline \Phi(\{v_1,v_3\})=\{v_1,v_3\},\\
&\overline \Phi(\{v_2,v_3\})=\{v_2,v_3\}, 
&\overline \Phi(\{v_0,v_1\})=0,\\
&\overline \Phi(\{v_0,v_2\})=\{v_1,v_2\},
&\overline \Phi(\{v_0,v_3\})=\{v_0,v_3\}-\{v_0,v_1\},\\
&\overline \Phi(\{v_0,v_1,v_2\})=0. 
\end{eqnarray*}
By a direct calculation $\overline \Phi^\infty=\overline \Phi$.  Hence (\ref{eq-5.1}) is given by
\begin{eqnarray*}
\overline \Phi^\infty\mid_{R(M(\overline f,\Delta\mathcal{H}))}: &&R(\{v_1\},\{v_2\},\{v_3\}, \{v_0,v_3\}, \{v_1,v_2\}, \{v_1,v_3\})\longrightarrow\\
&& R(\{v_1\},\{v_2\},\{v_3\}, \{v_0,v_3\}-\{v_0,v_1\}, \{v_1,v_2\}, \{v_1,v_3\})
\end{eqnarray*}
sending  $\{v_0,v_3\}$ to $ \{v_0,v_3\}-\{v_0,v_1\}$ and sending the other generators $\{v_1\}$, $\{v_2\}$, $\{v_3\}$, $\{v_1,v_2\}$, $\{v_1,v_3\}$ to themselves.  And $\pi_{M}\mid_{C_*^{\overline\Phi}(\Delta\mathcal{H};R)}$ in (\ref{eq-abc}) 
sends $ \{v_0,v_3\}-\{v_0,v_1\}$ to  $ \{v_0,v_3\}  $
and sends the other generators  to themselves.  Hence by (\ref{eq-6.v}), 
\begin{eqnarray*}
\tilde\partial_1\{v_0,v_3\}=\tilde\partial_1(\{v_1,v_3\})=v_3-v_1.
\end{eqnarray*}
Therefore,  by Theorem~\ref{th-6.5},  
\begin{eqnarray*}
H_0(\mathcal{H};R)={R}^{\oplus 2}, \text{\ \ \ } H_1(\mathcal{H};R)={R}, \text{\ \ \ } H_2(\mathcal{H};R)=0. 
\end{eqnarray*}
\end{example}

\begin{figure}[!htbp]
 \begin{center}
\begin{tikzpicture}[line width=1.5pt]

\coordinate [label=right:$v_0$]    (A) at (1.5,2); 
 \coordinate [label=left:$v_1$]   (B) at (0.9,0); 
 \coordinate  [label=right:$v_2$]   (C) at (2.4,0); 
\coordinate  [label=left:$v_3$]   (D) at (0,2);

\fill (1.5,2) circle (2.5pt);
\fill (0.9,0) circle (2.5pt);
\fill (2.4,0) circle (2.5pt);
\fill (0,2) circle (2.5pt);

 \coordinate[label=left:$\mathcal{H}$:] (G) at (-0.5,1);
 \draw [dashed,thick] (C) -- (B);
 \draw [dashed,thick] (C) -- (A);
 %\draw [dashed,thick] (C) -- (A);
  \draw [thick] (B) -- (D);
    \draw [thick] (B) -- (A);
    \draw [thick] (D) -- (A);

\fill [fill opacity=0.25][gray!100!white] (B) -- (A) -- (C) -- cycle;

 %\end{tikzpicture}
 
 %\bigskip

 %\begin{tikzpicture}

\coordinate [label=right:$v_0$]    (P) at (7.5,2); 
 \coordinate [label=left:$v_1$]   (Q) at (6.9,0); 
 \coordinate  [label=right:$v_2$]   (R) at (8.4,0); 
\coordinate  [label=left:$v_3$]   (S) at (6,2);

\draw[->] (7.5,2) -- (7.35,1.5);

\draw[->] (7.95,1) -- (7.95-1/3,0.9);

\fill (7.5,2) circle (2.5pt);
\fill (6.9,0) circle (2.5pt);
\fill (8.4,0) circle (2.5pt);
\fill (6,2) circle (2.5pt);

 \coordinate[label=left:$\Delta\mathcal{H}$ and $\overline V$:] (M) at (5.5,1);
 \draw [thick] (R) -- (Q);
 \draw [thick] (R) -- (P);
 %\draw [dashed,thick] (C) -- (A);
  \draw [thick] (Q) -- (S);
    \draw [thick] (P) -- (Q);
    \draw [thick] (S) -- (P);

\fill [fill opacity=0.25][gray!100!white] (R) -- (P) -- (Q) -- cycle;

 \end{tikzpicture}
\end{center}

\caption{Example~\ref{ex-6.8}.}
\end{figure}

%We can use % (\ref{eq-0.0a}) in Theorem~\ref{co-0.08} or  
%(\ref{eq-uuu})  in Corollary~\ref{co-6.888} to compute the embedded homology of hypergraphs satisfying condition (C). The following is a concrete example. 

%\begin{example}

%\end{example}

\section{Proofs  of Theorem~\ref{th-0.05} and Theorem~\ref{co-0.08}}\label{s7}

Theorem~\ref{th-0.05} is a slight generalization of Theorem~\ref{th-6.5}. The proof of Theorem~\ref{th-0.05} is an analogue of the proof of Theorem~\ref{th-6.5}. 

\begin{proof}[Proof of Theorem~\ref{th-0.05}]
Let $\mathcal{K}$ be a simplicial complex such that $\mathcal{H}\subseteq\mathcal{K}$. Let $\overline f_\mathcal{K}$ be a discrete Morse function on $\mathcal{K}$. By substituting $\Delta\mathcal{H}$ with $\mathcal{K}$, $\partial_*$ with $\partial_*^\mathcal{K}$, $\tilde\partial_*$ with $\tilde\partial_*^\mathcal{K}$,  and 
$\overline f$ with $\overline f_\mathcal{K}$ through  the related argument from Section~\ref{s3} to Section~\ref{s6},  Theorem~\ref{th-0.05} follows analogously with Theorem~\ref{th-6.5}. 
\end{proof}

%\begin{remark}
%We point out that both of the chain complexes
%\begin{eqnarray}
%R(M(\overline f_\mathcal{K}, \mathcal{K}))_k\cap \text{Inf}_k(R(\mathcal{H})_*)%&=&R(M(\overline f, \Delta\mathcal{H}))_k\cap \text{Inf}_k(R(\mathcal{H})_*)
%, \text{\ \ \ }k\geq 0, \label{eq-999}\\
%R(M(\overline f_\mathcal{K}, \mathcal{K}))_k\cap \text{Sup}_k(R(\mathcal{H})_*)%&=&R(M(\overline f, \Delta\mathcal{H}))_k\cap \text{Sup}_k(R(\mathcal{H})_*)
%,  \text{\ \ \ }k\geq 0 \label{eq-998}
%\end{eqnarray}
% depend on the choice of $\mathcal{K}$. However, by Theorem~\ref{th-0.05}, the homology groups of the two chain complexes (\ref{eq-999}) and (\ref{eq-998}) do not depend on the choice of $\mathcal{K}$. 
%\end{remark}

%\smallskip

Theorem~\ref{co-0.08} is a slight generalization of Corollary~\ref{co-6.888}. Similar with  Corollary~\ref{co-6.888},  Theorem~\ref{co-0.08} can be derived from Theorem~\ref{th-0.05} and Proposition~\ref{co-4.bb} (Remark~\ref{re-5.1}). 

\begin{proof}[Proof of Theorem~\ref{co-0.08}]
Let $\mathcal{H}$ be a hypergraph satisfying condition (C).  
Let $\mathcal{K}$ be a simplicial complex such that $\mathcal{H}\subseteq\mathcal{K}$. Let $g$ be a discrete Morse function on $\mathcal{H}$. By Lemma~\ref{pr-9.1},  there exists a discrete Morse function $\overline f$ on $\Delta\mathcal{H}$ such that both (i) and (ii) in Lemma~\ref{pr-9.1} are satisfied.  By \cite[Subsection~3.1]{h1},  $\Delta\mathcal{H}$ is a simplicial sub-complex of $\mathcal{K}$.  
By \cite[Lemma~4.2]{forman1}, the discrete Morse function $\overline f$ on $\Delta\mathcal{H}$ can be extended to be a discrete Morse function  $\overline f_\mathcal{K}$ on $\mathcal{K}$ such that  $\text{grad }(\overline f_\mathcal{K})$ is $\text{grad } \overline f$  on $\Delta\mathcal{H}$ and  $\text{grad } ( \overline f_\mathcal{K})$ is vanishing on $\mathcal{K}\setminus \Delta\mathcal{H}$.  With the help of Lemma~\ref{pr-9.1} (i) and (ii),  we have that $\text{grad }(\overline f_\mathcal{K})$ is $\text{grad } g$  on $\mathcal{H}$ and $\text{grad }(\overline f_\mathcal{K})$ is vanishing on $\mathcal{K}\setminus  \mathcal{H}$.  Therefore, each critical simplex of $\overline f_\mathcal{K}$ on $\mathcal{K}$ must be a critical hyperedge of $g$ on $\mathcal{H}$.  Similar with Corollary~\ref{co-6.888},  (\ref{eq-0.0a}) follows from Theorem~\ref{th-0.05}  and Proposition~\ref{co-4.bb} (Remark~\ref{re-5.1}).

We observe that
$
\partial^\mathcal{K}_*\mid_{\Delta\mathcal{H}}
$, the boundary map of $\Delta\mathcal{H}$, 
 does not depend on the choice of $\mathcal{K}$ as well as the choice of $\overline f_\mathcal{K}$.    Hence for each $k\geq 0$, 
 \begin{eqnarray}\label{eq-vvv}
 R(M(g,\mathcal{H}))_k\cap (\partial^\mathcal{K}_k)^{-1}(R(\mathcal{H})_{k-1})
 \end{eqnarray}
  does not depend on the choice of $\mathcal{K}$ as well as $\overline f_\mathcal{K}$.

 We also observe that $(\text{grad } \overline f_\mathcal{K})\mid_{\Delta\mathcal{H}}$ is determined by $\text{grad } g$, and  does not depend on the choice of $\mathcal{K}$ as well as   $\overline f_\mathcal{K}$.  Hence by letting $\Phi(\mathcal{K})$ be the discrete gradient flow  of $\overline f_\mathcal{K}$ on $\mathcal{K}$,  we have that $\Phi(\mathcal{K})\mid_{\Delta\mathcal{H}}$ is determined by $\text{grad } g$, and does not depend on the choice of $\mathcal{K}$ as well as   $\overline f_\mathcal{K}$.  Moreover,  by  \cite[Theorem~2.2]{forman3}, we have
\begin{eqnarray*}
\tilde \partial^\mathcal{K}_k=(\pi_{M}\mid_{C_*^{\Phi(\mathcal{K})}(\mathcal{K};R)})\circ \partial^{\mathcal{K}}_k\circ  (\Phi(\mathcal{K})^\infty\mid _{R(M( \overline f_\mathcal{K},\mathcal{K}))_*}) 
\end{eqnarray*}
and
\begin{eqnarray*}
\pi_{M}\mid_{C_*^{\Phi(\mathcal{K})}(\mathcal{K};R)}=(\Phi(\mathcal{K})^{\infty}\mid_{R(M( \overline f_\mathcal{K},\mathcal{K}))_*})^{-1}.  
\end{eqnarray*}
It follows  that restricted to (\ref{eq-vvv}), %$R(M(g,\mathcal{H}))_k\cap \partial_k^{-1}(R(\mathcal{H})_{k-1})$,
 $\tilde\partial^\mathcal{K}_k$ does not depend on the choice of $\mathcal{K}$ as well as $\overline f_\mathcal{K}$. 
 
 Therefore, the chain complex (\ref{eq-0.i}) does not depend on the choice of $\mathcal{K}$ as well as $\overline f_\mathcal{K}$. 
\end{proof}

% \begin{remark}
% By Theorem~\ref{co-0.08}, since the chain complex does not depend on the choice of $\mathcal{K}$ and the choice of $\overline f_\mathcal{K}$,   (\ref{eq-0.0a}) and (\ref{eq-uuu})  are equivalent for computations. 
% \end{remark}

\section{Morse Inequalities}\label{s8}

In this section, we give some Morse inequalities for hypergraphs  in Theorem~\ref{th-77.1} and Theorem~\ref{th-77.2}. 

\smallskip

Firstly,  we consider a chain complex 
\begin{eqnarray}\label{eq-7.1}
0\longrightarrow  C_n \overset{\partial_{n}}{\longrightarrow} C_{n-1}\overset{\partial_{n-1}}{\longrightarrow} \cdots \overset{\partial_2}{\longrightarrow} C_1 \overset{\partial_{1}}{\longrightarrow} C_0\longrightarrow 0
\end{eqnarray}
where for each $0\leq i\leq n$,  $C_i$ is a finite dimensional vector space over a field $\mathbb{F}$.  We use $c_i$ to denote the dimension of $C_i$.  Then
\begin{eqnarray}\label{eq-7.a}
c_i=\dim \text{Ker}\partial_i+\dim \text{Im}\partial_i. 
\end{eqnarray}
Let 
\begin{eqnarray*}
H_i(C_*)=\text{Ker}\partial_i/\text{Im}\partial_{i+1}. 
\end{eqnarray*}
The $i$-th Betti number is  
\begin{eqnarray}
b_i&=&\dim H_i(C_*)\nonumber\\
&=&\dim \text{Ker}\partial_i- \dim \text{Im} \partial_{i+1}.
\label{eq-7.b}
\end{eqnarray}
By a direct calculation of (\ref{eq-7.a}) and (\ref{eq-7.b}),  we have the following statements. 
\begin{enumerate}[(i).]
\item (The Weak Morse Inequalities). For every $N\geq 0$,
\begin{eqnarray}\label{eq-7.2}
c_N\geq b_N;
\end{eqnarray}

%\item 

\item (The Strong Morse Inequalities). For every $N\geq 0$,
\begin{eqnarray}\label{eq-7.99}
c_N-c_{N-1}+c_{N-2}-\cdots + (-1)^N c_0\geq b_N-b_{N-1}+b_{N-2}-\cdots+(-1)^N b_0;
\end{eqnarray}
moreover,
\begin{eqnarray}\label{eq-7.3}
c_0-c_1+c_2-c_3+\cdots+ (-1)^n c_n=b_0-b_1+b_2-b_3+\cdots + (-1)^n b_n. 
\end{eqnarray}
\end{enumerate}

Secondly,  let $\mathcal{H}$ be a hypergraph and let $\mathbb{F}$ be a field.  For each $n\geq 0$, let 
\begin{eqnarray*}
b_n&=&\dim H_n(\mathcal{H};\mathbb{F}),\\
r_n&=&\dim( \mathbb{F}(M(\overline f,\Delta\mathcal{H}))\cap \text{Inf}_n(\mathbb{F}(\mathcal{H})_*)),\\
R_n&=&\dim( \mathbb{F}(M(\overline f,\Delta\mathcal{H}))\cap \text{Sup}_n(\mathbb{F}(\mathcal{H})_*)). 
\end{eqnarray*}
The next theorem follows from Theorem~\ref{main-1}, Lemma~\ref{pr-5.1}  and (\ref{eq-7.2}) - (\ref{eq-7.3}). 

\begin{theorem}\label{th-77.1}
For any $N\geq 0$, 
%\begin{enumerate}[(i).]
%\item
\begin{eqnarray}\label{eq-nnn}
&R_N\geq r_N\geq b_N,\\
%\end{eqnarray}
%and 
%\item
%For any $N\geq 0$, 
%\begin{eqnarray}
\label{eq-m.1}
&\text{min}\Big\{\sum_{n= 0}^N (-1)^n r_{N-n}, \sum_{n= 0}^N (-1)^n R_{N-n}\Big\} \geq \sum_{n= 0}^N (-1)^n b_{N-n},\\
%\end{eqnarray}
%moreover,
%\begin{eqnarray}
\label{eq-m.2}
&\sum_{n\geq 0} (-1)^n b_n= \sum_{n\geq 0} (-1)^n r_n=  \sum_{n\geq 0} (-1)^n R_n. 
\end{eqnarray}
%\end{enumerate}
\end{theorem}

\begin{proof}
By Lemma~\ref{pr-5.1}, 
\begin{eqnarray*}
r_n&=&\dim(\overline\Phi^\infty \text{Inf}_n(\mathbb{F}(\mathcal{H})_*)),\\
R_n&=&\dim(\overline\Phi^\infty \text{Sup}_n(\mathbb{F}(\mathcal{H})_*)). 
\end{eqnarray*}
We notice $R_n\geq r_n$. 
By Theorem~\ref{main-1} and (\ref{eq-7.2}), we obtain (\ref{eq-nnn}).  By Theorem~\ref{main-1} and (\ref{eq-7.99}),  we obtain (\ref{eq-m.1}). By Theorem~\ref{main-1} and (\ref{eq-7.3}), we obtain (\ref{eq-m.2}). 
\end{proof}

Finally,  let $\mathcal{K}$ be a general simplicial complex such that $\mathcal{H}\subseteq\mathcal{K}$.  Let\begin{eqnarray*}
%b_n^{\mathcal{K}}&=&\dim H_n(\mathcal{H};\mathbb{F}),\\
r_n^\mathcal{K}&=&\dim( \mathbb{F}(M(\overline f,\mathcal{K}))\cap \text{Inf}_n(\mathbb{F}(\mathcal{H})_*)),\\
R_n^\mathcal{K}&=&\dim( \mathbb{F}(M(\overline f,\mathcal{K}))\cap \text{Sup}_n(\mathbb{F}(\mathcal{H})_*)). 
\end{eqnarray*}
The next theorem can be proved  analogously with Theorem~\ref{th-77.1}.

\begin{theorem}\label{th-77.2}
{\color{black} Suppose both $\text{Inf}_*(\mathcal{H})$ and $\text{Sup}_*(\mathcal{H})$ are $\text{grad}~f$-invariant. } 
For any $N\geq 0$ and any simplicial complex $\mathcal{K}$ such that $\mathcal{H}\subseteq\mathcal{K}$,  
%\begin{enumerate}[(i).]
%\item
%For any $N\geq 0$ and any simplicial complex $\mathcal{K}$ such that $\mathcal{H}\subseteq\mathcal{K}$, 
\begin{eqnarray}%\label{mmm}
&R_N^\mathcal{K}\geq r_N^\mathcal{K}\geq b_N,\nonumber\\ 
%\end{eqnarray*}
%and
%\item
%For any $N\geq 0$ and any simplicial complex $\mathcal{K}$ such that $\mathcal{H}\subseteq\mathcal{K}$, 
%\begin{eqnarray*}%\label{eq-m.111}
%\text{Inf}_{\mathcal{K}}\Big\{
&\text{min}\Big\{\sum_{n= 0}^N (-1)^n r^{\mathcal{K}}_{N-n}, \sum_{n= 0}^N (-1)^n R^\mathcal{K}_{N-n}\Big\}%\Big\} 
\geq \sum_{n= 0}^N (-1)^n b_{N-n},\nonumber\\
%\end{eqnarray*}
%here $\text{Inf}_{\mathcal{K}}$ is taken over all simplicial complexes $\mathcal{K}$ such that $\mathcal{H}\subseteq\mathcal{K}$; 
% moreover,  %for any simplicial complex $\mathcal{K}$ such that $\mathcal{H}\subseteq\mathcal{K}$,
%\begin{eqnarray}
\label{eq-m.222}
&\sum_{n\geq 0} (-1)^n b_n= \sum_{n\geq 0} (-1)^n r^\mathcal{K}_n=  \sum_{n\geq 0} (-1)^n R^\mathcal{K}_n. 
\end{eqnarray}
\qed
%\end{enumerate}
\end{theorem}

We observe that for each $n\geq 0$, both $r_n^\mathcal{K}$ and $R_n^\mathcal{K}$  depend on the choice of $\mathcal{K}$. However,  by (\ref{eq-m.222}) we see that both of the alternating sums $\sum_{n\geq 0} (-1)^n r^\mathcal{K}_n$ and  $ \sum_{n\geq 0} (-1)^n R^\mathcal{K}_n$ do not depend on the choice of $\mathcal{K}$. 

\section{Collapses and Embedded Homology}\label{s9}

In this section, we define the collapses of hypergraphs as a generalization of collapses of simplicial complexes. We prove in Theorem~\ref{pr-8.4} that the collapses preserve the embedded homology. 

\smallskip

Let $\mathcal{H}$ be a hypergraph.  Let  $\sigma^{(n)}<\tau^{(n+1)}$ be hyperedges of $\mathcal{H}$, $n\geq 0$,    such that  % Suppose
\begin{enumerate}[(1).]
\item
$\sigma$ is not a proper subset of any other hyperedges of $\mathcal{H}$;
\item
each proper subset $\eta$ of $\tau$ is a hyperedge of $\mathcal{H}$. 
\end{enumerate}
Let 
\begin{eqnarray}\label{eq-8.1}
\mathcal{H}'=\mathcal{H}\setminus\{\sigma,\tau\}. 
\end{eqnarray}
We say $\mathcal{H}$ {\it collapses onto} $\mathcal{H}'$ and call the operation (\ref{eq-8.1}) a {\it single elementary collapse}.  We notice that (1) implies that $\tau$ is a maximal hyperedge of $\mathcal{H}$. 

More generally, for two hypergraphs $\mathcal{H}$ and $\mathcal{H}'$, we say $\mathcal{H}$ collapses onto $\mathcal{H}'$ and write $\mathcal{H}\searrow \mathcal{H}'$,  if $\mathcal{H}$ can be transformed into $\mathcal{H}'$  by a finite sequence of single elementary collapses.

\begin{example}
Let $\mathcal{K}$ be a simplicial complex. Let $\mathcal{H}$ be a hypergraph. Suppose 
\begin{eqnarray*}
V(\mathcal{K})\cap V(\mathcal{H})=\{v_0,v_1,\ldots,v_n\}. 
\end{eqnarray*}
Then  the hypergraph $\mathcal{K}\cup \mathcal{H}$ has its  vertex-set $V(\mathcal{K})\cup V(\mathcal{H})$ and its hyperedge-set $\mathcal{K}\cup\mathcal{H}$.  If $\sigma^{(n)}<\tau^{(n+1)}$ are simplices of $\mathcal{K}$ such that
\begin{enumerate}[(i).]
\item
$\sigma<\tau$ give an elementary collapse of $\mathcal{K}$;
\item
for any $0\leq i\leq n$, $v_i\notin \sigma$,
\end{enumerate}
then $\sigma<\tau$ also give an elementary collapse of $\mathcal{K}\cup\mathcal{H}$. 
\end{example}

We prove some lemmas. 

\begin{lemma}\label{le-8.1}
Let $\mathcal{H}$ and $\mathcal{H}'$ be two hypergraphs.  Then 
%\begin{enumerate}[(i).]
%\item
%If $\mathcal{H}\searrow \mathcal{H}'$, then $\Delta\mathcal{H}\searrow \Delta\mathcal{H}'$ and $\delta\mathcal{H}\searrow \delta\mathcal{H}'$;
%\item
$\mathcal{H}\searrow \mathcal{H}'$  by a sequence of single elementary collapses iff. both of the followings hold
\begin{enumerate}[(i).]
\item
 both $\delta\mathcal{H}\searrow \delta\mathcal{H}'$  and $\Delta\mathcal{H}\searrow \Delta\mathcal{H}'$    
by the same sequence of single elementary collapses;
\item
$\mathcal{H}\setminus\mathcal{H}'=\Delta\mathcal{H}\setminus\Delta\mathcal{H}'=\delta\mathcal{H}\setminus\delta\mathcal{H}'$.  
\end{enumerate}
\end{lemma}

\begin{proof}
($\Longrightarrow$).  Suppose $\mathcal{H}\searrow \mathcal{H}'$. 
With out loss of generality, we can assume that $\mathcal{H}\searrow\mathcal{H}'$ by a single elementary collapse (\ref{eq-8.1}).  Then there exist hyperedges $\sigma^{(n)}<\tau^{(n+1)}$ of $\mathcal{H}$ such that both (1) and (2) are satisfied.  By (2), we have $\tau\in \delta\mathcal{H}$.  We observe that  by
removing $\sigma$ and $\tau$, we have    a single elementary collapse from $\Delta\mathcal{H}$ onto $\Delta\mathcal{H}'$, and also  a single elementary collapse from $\delta\mathcal{H}$ onto $\delta\mathcal{H}'$.  We obtain (i) and (ii).

($\Longleftarrow$). Suppose by a  sequence of single elementary collapses, both $\delta\mathcal{H}\searrow \delta\mathcal{H}'$  and $\Delta\mathcal{H}\searrow \Delta\mathcal{H}'$.  With out loss of generality,  we assume that the sequence is a single elementary collapse, i.e.  
\begin{eqnarray}\label{eq-8.111}
\Delta\mathcal{H}\searrow \Delta\mathcal{H}'  \text{ by removing certain } \sigma^{(n)}  \text{ and }\tau^{(n+1)},
\end{eqnarray}
 and  
 \begin{eqnarray}\label{eq-8.112}
 \delta\mathcal{H}\searrow \delta\mathcal{H}' \text{  by removing }\sigma\text{ and }\tau 
 \end{eqnarray}
  as well.   Suppose   $\mathcal{H}\setminus\mathcal{H}'=\Delta\mathcal{H}\setminus\Delta\mathcal{H}'=\delta\mathcal{H}\setminus\delta\mathcal{H}'$. Then we have (\ref{eq-8.1}).  
  By (\ref{eq-8.111}),  $\sigma$ is not a proper subset of any other simplices of $\Delta\mathcal{H}$, hence $\sigma$ is  not a proper subset of any other hyperedges of $\mathcal{H}$.  By (\ref{eq-8.112}), each proper subset $\eta$ of $\tau$ is a simplex of $\delta\mathcal{H}$, hence each $\eta$  is  a hyperedge of $\mathcal{H}$.     
   Thus both (1) and (2) are satisfied for $\sigma,\tau\in\mathcal{H}$. Hence $\mathcal{H}\searrow\mathcal{H}'$ by removing $\sigma$ and $\tau$. 
\end{proof}

\begin{lemma}\label{co-8.3}
Suppose $\mathcal{H}\searrow \mathcal{H}'$. Let $\overline{f}:\Delta\mathcal{H}'\longrightarrow \mathbb{R}$  be a discrete Morse function on $\Delta\mathcal{H}'$. %Let $c=\max_{\sigma\in \Delta\mathcal{H}'} f(\sigma)$. 
Then $\overline f$ can be extended to be  a discrete Morse function $\overline g$ on $\Delta\mathcal{H}$ % with $\Delta\mathcal{H}'=\Delta\mathcal{H}(c)$ and 
such that there is no critical simplices of $\overline g$  in $\Delta\mathcal{H}\setminus\Delta\mathcal{H}'$. 
\end{lemma}
\begin{proof}
By Lemma~\ref{le-8.1},  we have $\Delta\mathcal{H}\searrow \Delta\mathcal{H}'$.   
  Letting the simplicial complexes in \cite[Lemma~4.3]{forman1} be $\Delta\mathcal{H}$ and $\Delta\mathcal{H}'$, we obtain the lemma. 
\end{proof}

The next theorem follows from the above Lemma~\ref{le-8.1}, Lemma~\ref{co-8.3}  and Theorem~\ref{th-6.5}.

\begin{theorem}\label{pr-8.4}
Let $\mathcal{H}$ and $\mathcal{H}'$ be two hypergraphs. {\color{black} Suppose there exists a discrete Morse function $g$ on $\mathcal{H}$ such that   both $\text{Inf}_*(\mathcal{H})$ and $\text{Sup}_*(\mathcal{H})$ are $\text{grad}~g$-invariant, and both $\text{Inf}_*(\mathcal{H}')$ and $\text{Sup}_*(\mathcal{H}')$ are $\text{grad}~(g\mid_{\mathcal{H}'})$-invariant.  }  If $\mathcal{H}\searrow\mathcal{H}'$, then $H_*(\mathcal{H})\cong H_*(\mathcal{H}')$, $H_*(\Delta\mathcal{H})\cong H_*(\Delta\mathcal{H}')$, and $H_*(\delta\mathcal{H})\cong H_*(\delta\mathcal{H}')$. 
\end{theorem}

\begin{proof}
Without loss of generality, suppose $\mathcal{H}\searrow\mathcal{H}'$ by a single elementary collapse. Then we have (\ref{eq-8.1}) where $\sigma^{(n)}<\tau^{(n+1)}$ are hyperedges of $\mathcal{H}$  satisfying both (1) and (2). 
Moreover, 
$\Delta\mathcal{H}\searrow\Delta\mathcal{H}'$ by a single elementary collapse, i.e. 
\begin{eqnarray*}
\Delta\mathcal{H}'=\Delta\mathcal{H}\setminus \{\sigma,\tau\}
\end{eqnarray*}
 where $\sigma^{(n)}<\tau^{(n+1)}$  are simplices of $\Delta\mathcal{H}$ satisfying (i). $\sigma$ is not a proper subset of any other simplices of $\Delta\mathcal{H}$;
(ii). each proper subset $\eta$ of $\tau$ is a simplex of $\Delta\mathcal{H}$. 
%\end{enumerate} 

 By Corollary~\ref{cor1} and Lemma~\ref{co-8.3}, there exist  discrete Morse functions $\overline g$ on $\Delta\mathcal{H}$ and $\overline f$ on $\Delta\mathcal{H}'$  such that $\overline f=\overline g\mid _{\Delta\mathcal{H}'}$ and there is no critical simplices in $\Delta\mathcal{H}\setminus \Delta\mathcal{H}'$.  Hence
 \begin{eqnarray}\label{eq-8.8}
 M(\overline g,\Delta\mathcal{H})=M(\overline f,\Delta\mathcal{H}'). 
 \end{eqnarray}
On the other hand, for each $i\geq 0$, 
\begin{eqnarray}\label{eq-8.9}
\text{Sup}_i(R(\mathcal{H})_*)&=& R(\mathcal{H})_i +\partial_{i+1} \mathcal{R}(\mathcal{H})_{i+1}\nonumber\\
&=& R(\mathcal{H}')_i + R(\sigma,\tau)_i+\partial_{i+1} \mathcal{R}(\mathcal{H}')_{i+1}+ \partial_{i+1} (R(\sigma,\tau)_{i+1})\nonumber\\
&=& \text{Sup}_i(R(\mathcal{H}')_*)+ R(\sigma,\tau)_i+  \partial_{i+1} (R(\sigma,\tau)_{i+1}). 
\end{eqnarray}
Here in (\ref{eq-8.9}),  we use the notation
\begin{eqnarray*}  
R(\sigma,\tau)_*=\left\{
\begin{array}{cc}
R(\sigma),  & \text{\ \ if }*=n, \\
R(\tau),   & \text{\ \ if } *=n+1, \\
0, &\text{\ \  if }*\neq n, n+1.  
\end{array}
\right.
\end{eqnarray*}
We notice that
\begin{eqnarray}
&R(\sigma,\tau)\cap R(M(\overline f,\Delta\mathcal{H}'))=0, \label{eq-8.10}\\
&R(\partial_{n+1}\tau)\cap R(M(\overline f,\Delta\mathcal{H}'))=0. \label{eq-8.11}
\end{eqnarray}
Moreover, since each proper subset $\eta$ of $\sigma$ is also  a proper subset of $\tau$, it follows from (2) that each $\eta$ is a hyperedge of $\mathcal{H}$ as well as of $\mathcal{H}'$.  By taking $\eta$ as the faces of $\sigma$ of co-dimension $1$, we obtain
\begin{eqnarray}
&R(\partial_n\sigma)=R(\sum_{i=0}^n (-1)^i d_i\sigma)\subseteq R(\mathcal{H}')_{n-1}\subseteq \text{Sup}_{n-1}(R(\mathcal{H}')_*). \label{eq-8.12}
\end{eqnarray}
Here in (\ref{eq-8.12}), $d_i$ are the face maps deleting the $i$-th vertex of $\sigma$, $i=0,1,\ldots,n$.  
It follows from (\ref{eq-8.8}) - (\ref{eq-8.12}) that 
\begin{eqnarray*}
&&R (M(\overline g,\Delta\mathcal{H}))\cap \text{Sup}_i(R(\mathcal{H})_*)\\
&=&R(M(\overline f,\Delta\mathcal{H}'))\cap (\text{Sup}_i(R(\mathcal{H}')_*)+ R(\sigma,\tau)_i+  \partial_{i+1} (R(\sigma,\tau)_{i+1}))\\
&=&R(M(\overline f,\Delta\mathcal{H}'))\cap \text{Sup}_i(R(\mathcal{H}')_*). 
\end{eqnarray*}
Hence by Theorem~\ref{th-6.5},  we have  $H_*(\mathcal{H})\cong H_*(\mathcal{H}')$.

Moreover, it follows from Lemma~\ref{le-8.1}  and \cite{forman1} that $H_*(\Delta\mathcal{H})\cong H_*(\Delta\mathcal{H}')$ and $H_*(\delta\mathcal{H})\cong H_*(\delta\mathcal{H}')$.  
\end{proof}

\section{Level Hypergraphs}\label{s10}

%be the discrete Morse functions given by Proposition~\ref{pr1}.  
%Without loss of generality,  we may assume that $\overline f: \Delta\mathcal{H}\longrightarrow \mathbb{R}$ is injective. 
%Otherwise, we can perturb $\overline f$ slightly without changing the critical simplices and obtain a new discrete Morse function (cf.  \cite[p. 104]{forman1}).    We may also assume that $f$ and $\underline  f$  are injective as well.  
%  

In this section,  we define level hypergraphs and  give a collapse result  in Corollary~\ref{co-8.8}. 
\smallskip

 Let $\mathcal{H}$ be a hypergraph.  Let $f$ be a discrete Morse function on $\mathcal{H}$.     For each $c\in\mathbb{R}$, we define  the {\it level hypergraph} as the sub-hypergraph
\begin{eqnarray}\label{eq-uu}
\mathcal{H}[c]=\bigcup_{\substack{\sigma\in\mathcal{H} \text{ s.t.} \\ f(\sigma)\leq c }}\sigma
\end{eqnarray}
of $\mathcal{H}$.  
The associated simplicial complex of $\mathcal{H}[c]$ is
\begin{eqnarray}
\Delta(\mathcal{H}[c])&=&\bigcup_{\substack{\sigma\in\mathcal{H} \text{ s.t.} \\ f(\sigma)\leq c }} \bigcup_{\tau\in\Delta\sigma 
}\tau
\nonumber\\
&=&\bigcup_{\substack{\sigma\text{  is a maximal  hyperedge }\\ \text{ of }\mathcal{H}   \text{ s.t.}   f(\sigma)\leq c }} \bigcup_{\tau\in\Delta\sigma 
}\tau;
\label{eq-17.1}
\end{eqnarray}
and the lower-associated simplicial complex of $\mathcal{H}(c)$ is
\begin{eqnarray}\label{eq-17.4}
\delta(\mathcal{H}[c])=\bigcup_{\substack{ \sigma\in\delta\mathcal{H} \text{ s.t. for any }\\  \tau\in\Delta\sigma, f(\tau)\leq c   }} \bigcup_{\tau\in\Delta\sigma}\tau. 
\end{eqnarray}
Here $\Delta\sigma$ is defined in (\ref{eq-md1}).  
Particularly,  for a  simplicial complex $\mathcal{K}$, we point out that our level hypergraph $\mathcal{K}[c]$ may not be a simplicial complex, and  is different from   $\mathcal{K}(c)$ defined  in \cite[Definition~3.1]{forman1}.    In fact,  $\mathcal{K}(c)$ is the associated simplicial complex of  $\mathcal{K}[c]$.

Suppose $\overline f: \Delta\mathcal{H}\longrightarrow \mathbb{R}$ is a discrete Morse function on $\Delta\mathcal{H}$,  and the discrete Morse functions $f: \mathcal{H}\longrightarrow \mathbb{R}$  and $\underline  f: \delta\mathcal{H}\longrightarrow\mathbb{R}$ are given by $f=\overline f\mid_\mathcal{H}$ and $\underline{f}=\overline{f}\mid_{\delta{\mathcal{H}}}$ respectively.  
Letting the $CW$-complex $M$ in \cite[Definition~3.1]{forman1} be   $\Delta\mathcal{H}$,  we have
\begin{eqnarray}
(\Delta\mathcal{H})(c)&=&\bigcup_{\substack{\sigma\in\Delta\mathcal{H}  \text{ s.t.}  \\ \overline f(\sigma)\leq c }} \bigcup_{\tau\in\Delta \sigma%\atop \tau\neq\emptyset
}\tau\nonumber\\
&=&\bigcup_{\substack{\sigma\text{ is a maximal simplex }\\ \text{  of }\Delta\mathcal{H}    \text{ s.t.}   \overline f(\sigma)\leq c }} \bigcup_{\tau\in\Delta  \sigma%\atop\tau\neq\emptyset
}\tau. 
\label{eq-17.2}
\end{eqnarray}
We notice that for any $\sigma\in\Delta\mathcal{H}$, $\sigma$ is a maximal hyperedge of $\mathcal{H}$ iff. $\sigma$ is a maximal simplex of $\Delta\mathcal{H}$.  Hence  (\ref{eq-17.1}) and (\ref{eq-17.2}) give the same simplicial complex 
\begin{eqnarray}\label{eq-17.3}
\Delta(\mathcal{H}[c])=(\Delta\mathcal{H})(c):=\Delta\mathcal{H}[c]. 
\end{eqnarray}
Letting the $CW$-complex $M$ in \cite[Definition~3.1]{forman1} be   $\delta\mathcal{H}$,  we have
\begin{eqnarray}\label{eq-17.6}
(\delta\mathcal{H})(c)=\bigcup_{\substack{\sigma\in\delta\mathcal{H}  \text{ s.t.} \\  \underline  f(\sigma)\leq c}} \bigcup _{\tau\in\Delta\sigma
}\tau. 
\end{eqnarray}
By (\ref{eq-17.4}) and (\ref{eq-17.6}),  we have
\begin{eqnarray}\label{eq-17.7}
\delta(\mathcal{H}[c]) \subseteq (\delta\mathcal{H})(c). 
\end{eqnarray}
In particular, if all the simplices in $\delta\mathcal{H}$ are critical with respect to $\underline f$, then the equality in (\ref{eq-17.7}) holds.

%In particular, if $\mathcal{H}$ is a simplicial complex, % $\mathcal{K}$,  
%then (\ref{eq-17.3}) 
%is  defined in \cite[Definition~3.1]{forman1}.  In this case,  we point out that  (\ref{eq-17.3})  could be different from $\mathcal{H}(c)$ defined in (\ref{eq-uu}) for general discrete Morse functions $\overline f$. %, and is denoted as $M(c)$ there. 
%By Proposition~\ref{pr1}~(ii), we have a discrete Morse function $\underline  f: \delta\mathcal{H}\longrightarrow \mathbb{R}$. 

\smallskip

The next corollary is a consequence of Lemma~\ref{le-8.1} and \cite[Theorem~3.3]{forman1}.  

\begin{corollary}\label{co-8.8}
If $a<b$ are real numbers such that
\begin{enumerate}[(i).]
\item
 $[a,b]$ contains no critical values of $\overline f:\Delta\mathcal{H}\longrightarrow\mathbb{R}$;
 \item
  $\mathcal{H}[b]\setminus\mathcal{H}[a]=\Delta\mathcal{H}[b]\setminus \Delta\mathcal{H}[a]= \delta(\mathcal{H}[b])\setminus \delta(\mathcal{H}[a])$, 
  \end{enumerate}
  then $\mathcal{H}[b]\searrow\mathcal{H}[a]$. 
\end{corollary}

\begin{proof}
By (i) and \cite[Theorem~3.3]{forman1},  $\Delta\mathcal{H}[b]\searrow \Delta\mathcal{H}[a]$ by a sequence of single elementary collapses. We notice that $\delta(\mathcal{H}[b])$ (resp. $\delta(\mathcal{H}[a])$) is a simplicial subcomplex of $\Delta\mathcal{H}[b]$ (resp. $\Delta\mathcal{H}[a]$).  Hence by (ii),  $\delta(\mathcal{H}[b])\searrow \delta(\mathcal{H}[a])$ by the same  sequence of single elementary collapses.  Therefore, by Lemma~\ref{le-8.1},     $\mathcal{H}[b]\searrow\mathcal{H}[a]$. 
\end{proof}

\medskip

\noindent {\bf Acknowledgement}.  The authors would like to thank   Prof. Yong Lin  for his supports,  discussions and encouragements. The authors would like to express their deep gratitude to the reviewer(s) for their careful reading, valuable comments, and helpful suggestions. %The project was supported in part by the Singapore Ministry of Education research grant

\bigskip

 Shiquan Ren  %(for correspondence) 
 
 Address:  
Yau  Mathematical Sciences Center, Tsinghua University,  China 100084.  

  e-mail:  %rsq19@tsinghua.org.cn 
  srenmath@126.com
  
  \medskip
  
  Chong Wang
  
  Address: School of Information, Renmin University of China, China 100872. 
  
  e-mail:  wangchong\_618@163.com 
  
  \medskip

Chengyuan Wu
  
Address: Department of Mathematics, National University of Singapore, Singapore 119076. 
  
  e-mail: wuchengyuan@u.nus.edu
  
  \medskip
  
Jie Wu
  
Address: School of Mathematics and Information Science,  Hebei Normal University, China 050024.   %Department of Mathematics, National University of Singapore. 119076, Singapore. 
  
  e-mail: wujie@hebtu.edu.cn

%\newpage

   %To study the evolution of $\mathcal{H}(c)$  as $c$ increases, we only need to investigate the case
%\begin{eqnarray*}
%\mathcal{H}(b)=\mathcal{H}(a)\cup \{\sigma\},   \text{\ \ \ }\sigma\in\mathcal{H}  
%\end{eqnarray*}
%for $a<b$. 

%By \cite[Lemma~8.1]{forman1},  for any critical simplices $\sigma,\tau\in M(\overline f, \Delta\mathcal{H})$,  if $\sigma\neq\tau$,  then
%\begin{eqnarray*}
%\langle \overline\Phi^\infty\sigma,\tau\rangle=0. 
%\end{eqnarray*}
 
 \end{document}